\documentclass[11pt]{amsart}

\usepackage{epigamath}

\usepackage[english]{babel}

\usepackage{xypic}
\usepackage{dsfont}

\newtheorem{theo}{Theorem}[section]
\newtheorem{lemma}[theo]{Lemma}
\newtheorem{prop}[theo]{Proposition}
\newtheorem{cor}[theo]{Corollary}
\newtheorem{defi}[theo]{Definition}

\newtheorem*{conjnonumber}{Conjecture}

\theoremstyle{definition}
\newtheorem{ex}[theo]{Example}

\newtheorem{rem}[theo]{Remark}

\newcommand{\f}{\phi}

\newcommand{\ida}{\mathfrak{a}}

\newcommand{\spec}{\operatorname{Spec}}

\newcommand{\Gl}{\operatorname{GL}}

\newcommand{\Aut}{\operatorname{Aut}}
\newcommand{\Hom}{\operatorname{Hom}}

\newcommand{\id}{\operatorname{id}}

\newcommand{\N}{\mathcal{N}}

\newcommand{\nn}{\mathbb{N}}

\newcommand{\CC}{\mathcal{C}}

\newcommand{\I}{\mathbb{I}}

\newcommand{\Gm}{\mathbb{G}_m}

\newcommand{\Ga}{\mathbb{G}_a}
\def\Sl{\operatorname{SL}}

\newcommand{\red}{\operatorname{red}}
\newcommand{\kb}{\overline{k}}
\newcommand{\Rep}{\operatorname{Rep}}

\newcommand{\rank}{\operatorname{rank}}
\newcommand{\ir}{\operatorname{rank}}

\EpigaVolumeYear{4}{2020}
\EpigaArticleNr{1}
\ReceivedOn{August 31, 2019}
\InFinalFormOn{November 22, 2019}
\AcceptedOn{January 2, 2020}

\title{Free Proalgebraic Groups}
\titlemark{Free Proalgebraic Groups}

\author{Michael Wibmer}
\address{Institute of Analysis and Number Theory,
Graz University of Technology,
Kopernikusgasse 24/II,
8010 Graz, Austria}
\email{wibmer@math.tugraz.at}

\authormark{M. Wibmer}

\AbstractInEnglish{Replacing finite groups by linear algebraic groups, we study an algebraic-geometric counterpart of the theory of free profinite groups. In particular, we introduce free proalgebraic groups and characterize them in terms of embedding problems. The main motivation for this endeavor is a differential analog of a conjecture of Shafarevic.}

\MSCclass{14L15; 14L17; 34M50; 12H05}

\KeyWords{Affine group schemes; proalgebraic groups; differential Galois theory; embedding problems}

\TitleInFrench{Groupes proalg\'ebriques libres}

\AbstractInFrench{En rempla\c{c}ant les groupes finis par les groupes alg\'ebriques lin\'eaires, nous \'etudions une contrepartie alg\'ebro-g\'eom\'etrique de la th\'eorie des groupes profinis libres. En particulier, nous introduisons les groupes proalg\'ebriques libres et nous les caract\'erisons en termes de probl\`emes de plongements. Ce travail est principalement motiv\'e par un analogue diff\'erentiel d'une conjecture de Shafarevic.}

\acknowledgement{This work was supported by the NSF grants DMS-1760212, DMS-1760413, DMS-1760448 and the Lise Meitner grant M 2582-N32 of the Austrian Science Fund FWF.}

\begin{document}

%%%%%%%%%%%%%%%%%%%%%%%%%%%%%%%
% Title page
%%%%%%%%%%%%%%%%%%%%%%%%%%%%%%%

%\removeabove{}
%\removebetween{}
%\removebelow{}

\maketitle

\begin{prelims}

\DisplayAbstractInEnglish

\bigskip

\DisplayKeyWords

\medskip

\DisplayMSCclass

\bigskip

\languagesection{Fran\c{c}ais}

\bigskip

\DisplayTitleInFrench

\medskip

\DisplayAbstractInFrench

\end{prelims}

%%%%%%%%%%%%%%%%%%%%%
% Table of Contents
%%%%%%%%%%%%%%%%%%%%%

\newpage

\setcounter{tocdepth}{2} 

\tableofcontents

%%%%%%%%%%%%%%%%%%%%%
% Content begins here
%%%%%%%%%%%%%%%%%%%%%

\section{Introduction}

%?? simplify proof of existence using (vii)??

Free profinite groups play an important role in the theory of profinite groups and in Galois theory. For example, the Shafarevic conjecture in inverse Galois theory states that the absolute Galois group of the maximal abelian extension $\mathbb{Q}^{\operatorname{ab}}$ of $\mathbb{Q}$ is the free profinite group on a countably infinite set.  (See e.g., \cite{Harbater:ShafarevicConjecture}.)  A solvable version of Shafarevic's conjecture was established by Iwasawa in \cite{Iwasawa:OnSolvableExtensionsOfAlgebraicNumberFields} and a geometric version of Shafarevic's conjecture is known to be true:

\begin{theo}[Douady, Pop, Harbater] \label{theo: Douady Pop Harbater}
    Let $k$ be an algebraically closed field. Then the absolute Galois group of $K=k(x)$, the field of rational functions over $k$, is the free profinite group on a set of cardinality $|K|$.
\end{theo}
This was first proved by Douady in characteristic zero (\cite{Douady:DeterminationDunGroupeDeGalois}) and then later, independently by Pop (\cite{Pop:EtaleGaloisCoversOfAffineSmoothCurves}) and Harbater (\cite{Harbater:FundamentalGroupsAndEmbeddingProblemsInCharacteristicp}) in positive characteristic.

Differential Galois theory (\cite{SingerPut:differential}, \cite{Magid:LecturesOnDifferentialGaloisTheory}) is an extension of classical Galois theory: Instead of polynomials over a field, one considers linear differential equations over a differential field. The differential Galois group of a linear differential equation is a linear algebraic group. Accordingly, the absolute differential Galois group of a differential field is a projective limit of linear algebraic groups, i.e., an affine group scheme or a proalgebraic group.

A basic example of a differential field is the field $k(x)$ with the derivation $\frac{d}{dx}$. No explicit description of the absolute differential Galois group of $k(x)$ is presently known. However, around the turn of the century, Matzat promoted the idea that their should be a differential analog of Theorem \ref{theo: Douady Pop Harbater}, i.e., a differential version of Shafarevic's conjecture:

\begin{conjnonumber}[Matzat]
    Let $k$ be an algebraically closed field of characteristic zero. Then the absolute differential Galois group of $K=k(x)$ is the free proalgebraic group on a set of cardinality $|K|$.
\end{conjnonumber}
Matzat's conjecture implies Theorem \ref{theo: Douady Pop Harbater} (in characteristic zero) and is a far-reaching generalization of the solution of the inverse problem of differential Galois theory over $k(x)$, which asserts that every linear algebraic group is a differential Galois group over $k(x)$ (\cite{Hartmann:OnTheInverseProblemInDifferentialGaloisTheory}). For example, Matzat's conjecture implies that for a fixed non-trivial linear algebraic group $G$, the set of isomorphism classes of Picard-Vessiot extensions with differential Galois group isomorphic to $G$, has cardinality $|K|$, while the solution of the inverse problem merely asserts that this set is non-empty.

%
%
%
%Theorem \ref{theo: Douady Pop Harbater} can be seen as a geometric version of Shafarevic's conjecture, while Matzat's conjecture can be seen

An embedding problem (sometimes also called a lifting problem in the literature) for a profinite group $\Gamma$ consist of two epimorphisms $G\twoheadrightarrow H$ and $\Gamma\twoheadrightarrow H$ of profinite groups. A solution is an epimorphism $\Gamma\twoheadrightarrow G$ such that
\begin{equation} \label{eqn: embedding problem introduction}
    \xymatrix{
    \Gamma \ar@{->>}[rd] \ar@{..>>}[d]\\
    G \ar@{->>}[r] & H
}
\end{equation}
commutes. The proof of Theorem \ref{theo: Douady Pop Harbater} has two main steps:
\begin{enumerate}
    \item The Chatzidakis-Melnikov characterization of free profinite groups in terms of embedding problems (\cite[Prop. 3.5.11]{RibesZalesskii:ProfiniteGroups}): A profinite group $\Gamma$ of infinite rank $\kappa$ is free on a set of cardinality $\kappa$ if and only if every non-trivial embedding problem (\ref{eqn: embedding problem introduction}) with $G$ a finite group has $\kappa$ different solutions.
    \item Establishing that the absolute Galois group of $k(x)$ satisfies the above characterization.
\end{enumerate}

 \renewcommand{\labelenumi}{{\rm (\roman{enumi})}}

Over the last decade, progress towards Matzat's conjecture, along the lines of the above sketch, has been hampered by two aspects: Firstly, there has been no clear definition of what a free proalgebraic group should be, and accordingly, no characterization in terms of embedding problems has been available. Secondly, despite some positive results (\cite{Kovacic:TheInverseProblemInTheGaloisTheoryOfDifferentialFields},\cite{Kovacic:OnTheInverseProblmelII}, \cite{MatzatPut:ConstructiveDifferentailGaloisTheory},\cite{Oberlies:EinbettungsproblemeInDerDifferentialGaloisTheorie},\cite{Hartmann:OnTheInverseProblemInDifferentialGaloisTheory}), not enough has been known about the solvability of embedding problems for the absolute differential Galois group of $k(x)$.

In this article we introduce free proalgebraic groups and we characterize them in terms of embedding problems. In other words, we complete the first step of the two steps towards Matzat's conjecture. Paired with recent progress (\cite{BachmayrHarbaterHartmannWibmer:DifferentialEmbeddingProblems},\cite{BachmayrHartmannHarbaterPopLarge},\cite{BachmayrHarbaterHartmann:DifferentialEmbeddingProblemsOverLaurentSeriesFields}) on the solvability of differential embedding problems, this yields a special case of Matzat's conjecture that will appear in \cite{BachmayrHarbaterHartmannWibmer:FreeDifferentialGaloisGroups}.

\begin{theo}[{\cite[Theorem 3.10]{BachmayrHarbaterHartmannWibmer:FreeDifferentialGaloisGroups}}] \label{theo: Intro Matzat special case}
    Matzat's conjecture is true if the field of constants $k$ is countable and of infinite transcendence degree over $\mathbb{Q}$. In other words, Matzat's conjecture is true for the field $k=\overline{\mathbb{Q}(y_1,y_2,\ldots)}$, the algebraic closure of the field of rational functions in countably many variables over $\mathbb{Q}$.
\end{theo}

The restriction that $k$ is countable enters the picture because over a countable field our characterization of free proalgebraic groups on a countably infinite set is particularly simple. (See Theorem \ref{theo: Intro Iwasawa} below.) However, we do provide a characterization (in fact several equivalent ones) of free proalgebraic groups over fields of arbitrary cardinality. In \cite{BachmayrHarbaterHartmannWibmer:FreeDifferentialGaloisGroups} these equivalent characterizations of free proalgebraic groups in terms of embedding problems are translated to equivalent characterizations of differential fields with a free absolute differential Galois group in terms of differential embedding problems. We hope that this characterization and
work in progress %(\cite{BachmayrHarbaterHartmannWibmer:TheDifferentialGaloisGroupOfRationalFunctionField})
 on differential embedding problems will lead to a proof of Matzat's conjecture for uncountable fields $k$, in particular, for the classical case $k=\mathbb{C}$.

\vspace{5mm}
%
%The main body of this article is not about differential Galois theory and Matzat's conjecture,

In this article we introduce two classes of proalgebraic groups and we answer the question when these two notions coincide.
% that we introduce, namely, free proalgebraic groups and saturated proalgebraic groups.
%and the question when these two notions coincide.
Let us discuss the two classes individually:

%\vspace{5mm}
%
%\emph{Free proalgebraic groups}

\subsection*{Free proalgebraic groups}

We introduce free proalgebraic groups through a universal mapping property: The free proalgebraic group on a set $X$ over a field $k$ is a proalgebraic group $\Gamma(X)$ over $k$ together with a map $X\to \Gamma(X)(\kb)$ that is universal among all maps satisfying the following property: Every morphism from $\Gamma(X)$ to an algebraic group maps almost all elements of $X$ to the identity. To establish the existence of free proalgebraic groups we use the tannakian machinery (\cite{DeligneMilne:TannakianCategories}, \cite{Deligne:categoriestannakien}).

As in the theory of profinite groups, where one considers, e.g., pro-solvable groups or pro-$p$-groups, it is often convenient or interesting to restrict attention to a certain class $\CC$ of linear algebraic groups, for example the class of all reductive algebraic groups, or the class of all unipotent algebraic groups. As long as the class $\CC$ satisfies certain closure properties there is a well-behaved notion of a pro-$\CC$-group and indeed we introduce free proalgebraic groups in this context, i.e., we define free pro-$\CC$-groups.

Our notion of free pro-$\CC$-groups encompasses three previously known constructions. Firstly, if $\CC$ is a class of finite constant algebraic groups corresponding to a class $\mathtt{C}$ of finite groups, then a free pro-$\CC$-group can be identified with a profinite group, indeed with a free pro-$\mathtt{C}$-group. In this spirit, most of our results contain familiar results about profinite groups as a special case.

Secondly, the free proalgebraic group on a finite set $X$ is the proalgebraic completion of the (abstract) free group on $X$.

Thirdly, free prounipotent groups have already been introduced in \cite{LubotzkyMagid:CohomologyOfUnipotentAndPropunipotentGroups}, where they were studied in connection with the cohomology of unipotent groups and characterized as the prounipotent groups of cohomological dimension one. Their internal structure has further been investigated in \cite{LubotzkyMagid:FreeProunipotentGroups}. Moreover, in \cite{Magid:ThePicardVessiotAntiderivativeClosure}, a unipotent version of Matzat's conjecture over $\mathbb{C}$ has been proved: The differential Galois group of the maximal prounipotent Picard-Vessiot extension of $\mathbb{C}(x)$, is a free prounipotent group.

We note that free profinite groups play a distinguished role in the theory of profinite groups (\cite{RibesZalesskii:ProfiniteGroups}, \cite{FriedJarden:FieldArithmetic}) and that there are several results about free profinite groups that we do not touch upon in this paper that suggest themselves for a generalization to free proalgebraic groups.

\subsection*{Saturated proalgebraic groups}

Roughly speaking, saturated proalgebraic groups behave like universal objects in the category of proalgebraic groups. Our defintion of saturated proalgebraic groups gives a precise meaning to the vague idea of forming the projective limit of all linear algebraic groups, or of forming the projective limit of all linear algebraic groups belonging to a certain class $\CC$.

Every proalgebraic group $\Gamma$ can canonically be written as a projective limit of linear algebraic groups. Indeed, $\Gamma=\varprojlim_{N}\Gamma/N$, where $N$ ranges over all normal closed subgroups of $\Gamma$ such that $\Gamma/N$ is algebraic.
There is no difficulty in constructing a proalgebraic group $\Gamma$ such that, up to isomorphism, all linear algebraic groups occur as quotients of $\Gamma$. More generally, one may construct a proalgebraic group $\Gamma$ such that, up to isomorphism, all epimorphisms of linear algebraic groups occur in the canonical projective system for $\Gamma$. This property does not determine $\Gamma$ up to isomorphism. So one may make the stronger requirement, that for every algebraic quotient $H$ of $\Gamma$ every epimorphism $G\twoheadrightarrow H$ of linear algebraic groups occurs, up to isomorphism, in the canonical projective system for $\Gamma$, i.e., we ask that the embedding problem
$$
\xymatrix{
    \Gamma \ar@{->>}[rd] \ar@{..>>}[d]\\
    G \ar@{->>}[r] & H
}
$$
for $\Gamma$ has a solution. Even this property does not determine $\Gamma$ up to isomorphism. However, it turns out that, if $k$ is countable and we require that $\Gamma$ has countable rank, i.e., $\Gamma$ is the projective limit of linear algebraic groups, where the projective limit is taken over a countable set, then there exists, up to isomorphism, a unique such $\Gamma$. Moreover, if $k$ has characteristic zero, $\Gamma$ is isomorphic to the free proalgebraic group on a countably infinite set. To summarize, we have:

\begin{theo}[Corollary \ref{cor: Iwasawa for algebaic groups}] \label{theo: Intro Iwasawa}
    Let $k$ be a countable field of characteristic zero and $\Gamma$ a proalgebraic group over $k$ of countable rank. Then $\Gamma$ is a free proalgebraic group on a countably infinite set if and only if for every epimorphism $G\twoheadrightarrow H$ of linear algebraic groups and every epimorphism $\Gamma\twoheadrightarrow H$ the embedding problem
    $$
    \xymatrix{
        \Gamma \ar@{->>}[rd] \ar@{..>>}[d]\\
        G \ar@{->>}[r] & H
    }
    $$
    has a solution.
\end{theo}

Theorem \ref{theo: Intro Iwasawa} can be seen as an algebraic-geometric version of Iwasawa's freeness theorem (\cite[Theorem 4]{Iwasawa:OnSolvableExtensionsOfAlgebraicNumberFields}, \cite[Cor. 3..5.10]{RibesZalesskii:ProfiniteGroups}) which states that a profinite group of countably infinite rank is free on a countably infinite set if and only if every embedding problem
$$
\xymatrix{
    \Gamma \ar@{->>}[rd] \ar@{..>>}[d]\\
    G \ar@{->>}[r] & H
}
$$
where $G$ is a finite group, is solvable. In fact, we establish a relative version or $\CC$-version  of Theorem~\ref{theo: Intro Iwasawa} so that indeed Iwasawa's result is a special case of ours. Theorem \ref{theo: Intro Iwasawa} is exactly the group theoretic input needed for the proof of the special case of Matzat's conjecture (Theorem~\ref{theo: Intro Matzat special case}).

However, we do provide a version of Theorem \ref{theo: Intro Iwasawa} that applies for base fields of arbitrary cardinality, that we will discuss now. The rank of a proalgebraic group $G$ can be defined as the smallest cardinal $\kappa$ such that $G$ is the projective limit of linear algebraic groups, taken over a directed set of cardinality $\kappa$. A pro-$\CC$-group $\Gamma$ is saturated if for all epimorphisms $G\twoheadrightarrow H$, $\Gamma\twoheadrightarrow H$ of pro-$\CC$-groups with $\rank(H)<\rank(\Gamma)$ and $\rank(G)\leq\rank(\Gamma)$ the embedding problem
$$
\xymatrix{
    \Gamma \ar@{->>}[rd] \ar@{..>>}[d]\\
    G \ar@{->>}[r] & H
}
$$
has a solution. Morally, a proalgebraic group is saturated if it solves as many embedding problems as possible. In Theorem \ref{theo: saturated} we provide several equivalent characterizations of this idea. The general version of Theorem \ref{theo: Intro Iwasawa} states the following:

\begin{theo}[Theorem \ref{theo: free=saturated}] \label{theo: Intro main}
    Let $k$ be a field of characteristic zero and let $\Gamma$ be a pro-$\CC$-group whose rank is such that $\rank(\Gamma)=\kappa\geq|k|$. Then $\Gamma$ is the free pro-$\CC$-group on a set of cardinality $\kappa$ if and only if $\Gamma$ is saturated.
\end{theo}
We show by example that the assumption $\rank(\Gamma)\geq |k|$ in Theorem \ref{theo: Intro main} is necessary.

\vspace{5mm}

We conclude this introduction with a brief outline of the article: After introducing some basic terminology and constructions, the existence of free pro-$\CC$-groups is established in Section 2 using the tannakian machinery. We also compute some examples of free pro-$\CC$-groups. In Section 3 saturated pro-$\CC$-groups are introduced and several equivalent characterizations of saturated pro-$\CC$-groups are established. We then discuss the questions of existence and uniqueness of saturated pro-$\CC$-groups: Saturated pro-$\CC$-groups of a fixed rank are unique up to isomorphism. However, their existence can only be guaranteed for a rank at least the cardinality of the base field. We also show that, over a field of characteristic zero, a free pro-$\CC$-group $\Gamma$ on a set of cardinality $\kappa$ is saturated if $\rank(\Gamma)=\kappa$. Paired with the uniqueness of saturated pro-$\CC$-groups, this yields Theorem \ref{theo: Intro main}.

%Finally, in Section 4 (Matzat's conjecture) we translate our results about saturated pro-$\CC$-groups and embedding problems to differential embedding problems. We deduce Theorem \ref{theo: Intro Matzat special case} from Theorem \ref{theo: Intro Iwasawa} and a result about differential embedding problems from \cite{BachmayrHartmannHarbaterPopLarge}.
%We also briefly discuss a version of Matzat's conjecture in positive characteristic.

\vspace{5mm}

The author is grateful to Annette Bachmayr, Zo\'{e} Chatzidakis, David Harbater, Julia Hartmann, Andy Magid and Anand Pillay for helpful discussions during the preparation of this manuscript. %The author also would like to thank the anonymous referee for helpful comments.

\section{Free pro-$\CC$-groups}

The main goal of this section is to establish the existence of free pro-$\CC$-groups (Definition \ref{defi: free proalgebraic group}). We first need to introduce some terminology.

\subsection{Preliminaries and notation}

All rings are assumed to be commutative and unital. We work over a fixed base field $k$ throughout. The algebraic closure of $k$ is denoted by $\kb$. We assume some familiarity with the theory of linear algebraic groups and affine group schemes. We recall some basic results and notions that will be used throughout the text below. (See e.g., \cite{Milne:AlgebraicGroupsTheTheoryOfGroupSchemesOfFiniteTypeOverAField}, \cite{DemazureGabriel:GroupesAlgebriques} or \cite{Waterhouse:IntroductiontoAffineGroupSchemes}.)

To avoid endless repetitions of the word affine we use the term \emph{algebraic group} to mean affine group scheme of finite type over $k$. In particular, over a field of positive characteristic, an algebraic group need not be reduced. A \emph{proalgebraic group} is, by definition, an affine group scheme over $k$.  This terminology emphasizes the approach of this paper and is justified by the fact that the proalgebraic groups are precisely the projective limits of algebraic groups. Projective limits are always assumed to be taken over a directed partially ordered set. Projective limits exist in the category of proalgebraic groups and they can indeed be taken pointwise, i.e., $(\varprojlim G_i)(R)=\varprojlim G_i(R)$ for any $k$-algebra $R$.

 A \emph{closed subgroup} of a proalgebraic group is a closed subgroup scheme.
Let $G$ be a proalgebraic group. A closed normal subgroup $N$ of $G$ is called \emph{coalgebraic} if the quotient $G/N$ is algebraic, i.e., of finite type. The isomorphism theorems from the theory of (abstract) groups hold verbatim for proalgebraic groups.

If $X$ is an affine scheme, we write $X_{\operatorname{red}}$ for the underlying reduced scheme. Similarly, if $R$ is a ring, we write $R_{\operatorname{red}}$ for the quotient of $R$ by its nilradical.

If $X$ is an affine scheme over $k$, we write $k[X]$ for the $k$-algebra of global sections of $X$, so $X=\spec(k[X])$. If $\f\colon X\to Y$ is a morphism of affine schemes over $k$, we denote the scheme-theoretic image of $\f$ by $\f(X)$, i.e., $\f(X)$ is the smallest closed subscheme of $Y$ such that $\f$ factors through the inclusion $\f(X)\hookrightarrow Y$. In terms of ideals, the defining ideal of $\f(X)$ is simply the kernel of the dual map $\f^*\colon k[Y]\to k[X]$. If $\f\colon G\to H$ is a morphism of proalgebraic groups, then $\f(G)$ is a closed subgroup of $H$.
For a morphism $\f\colon G\to H$ of proalgebraic groups the following conditions are equivalent:
\begin{itemize}
    \item $\f(G)=H$.
    \item The dual map $\f^*\colon k[H]\to k[G]$ is injective.
    \item There exists a normal closed subgroup $N$ of $G$ such that $H\simeq G/N$ and $\f$ identifies with the canonical map $G\to G/N$.
    \item For every $k$-algebra $R$ and $h\in H(R)$, there exists a faithfully flat $R$-algebra $S$ and an element $g\in G(S)$ such that $g$ maps to $h\in H(R)\subseteq G(S)$ under $\f$.
\end{itemize}
If these conditions are satisfied we call $\f$ an \emph{epimorphism} and indicate this by writing $G\twoheadrightarrow H$. If $G\twoheadrightarrow H$ is an epimorphism of algebraic groups, then $G(\kb)\to H(\kb)$ is a surjective map. (The converse is true if $H$ is smooth.)

We will often treat a proalgebraic group $G$ as a functor from the category of $k$-algebras to the category of groups. For a $k$-algebra $R$ we then often identify $G(R)$ with $\Hom(k[G],R)$, the set of $k$-algebra homomorphisms from $k[G]$ to $R$. For a morphism $\f\colon G\to H$ of proalgebraic groups we denote with $\f_R\colon G(R)\to H(R)$ the corresponding homomorphism of groups. If confusion is unlikely, we may write $\f(g)$ instead of $\f_R(g)$ for $g\in G(R)$.

To define the structure of a proalgebraic group on an affine scheme $G$ over $k$ is equivalent to defining the structure of a Hopf algebra on $k[G]$. The normal closed subgroups of $G$ are in one-to-one correspondence with the Hopf subalgebras of $k[G]$. In more detail, a normal closed subgroup $N$ of $G$ corresponds to the image of $k[G/N]\to k[G]$, the morphism of Hopf algebras, dual to the canonical map $G\to G/N$. We will usually identify $k[G/N]$ with this Hopf subalgebra of $k[G]$. For normal closed subgroups $N_1,N_2$ of $G$, we have $k[G/N_1]\cap k[G/N_2]=k[G/N_1N_2]$ and $k[G/N_1]\cdot k[G/N_2]=k[G/(N_1\cap N_2)]$.

An algebraic group $G$ is \emph{finite} if $k[G]$ is a finite dimensional vector space over $k$. The identity component of an algebraic group $G$ is denoted by $G^o$. This notion extends to proalgebraic groups: If $G=\varprojlim_{i\in I} G_i$ is the projective limit of algebraic groups $G_i$, then $G^o=\varprojlim_{i\in I} G_i^o$.

For an abelian group $M$, we denote by $D(M)$ the functor from the category of $k$-algebras to the category of groups given by
$$D(M)(R)=\Hom(M,R^\times)$$
for any $k$-algebra $R$. Here $\Hom(M,R^\times)$ denotes the set of homomorphism of abelian groups from $M$ to $R^\times$. Then $D(M)$ is representable, i.e., a proalgebraic group. Indeed, $D(M)$ is represented by $k[D(M)]=kM$, the group algebra of $M$ over $k$. A proalgebraic group isomorphic to some $D(M)$ is called \emph{diagonalizable}.
The (contravariant) functor $M\rightsquigarrow D(M)$ is exact and fully faithful.

Any finite group $\mathtt{G}$ can naturally be interpreted as an algebraic group over $k$. This algebraic group is often called the \emph{constant} group scheme associated with $\mathtt{G}$, we denote it by $\mathtt{G}_k$.
The functor $\mathtt{G}\rightsquigarrow\mathtt{G}_k$ from the category of finite groups to the category of algebaic groups is fully faithful and extends to a fully faithful functor from the category of profinite groups to the category of proalgebraic groups: If $\mathtt{G}=\varprojlim_{i\in I}\mathtt{G}_i$ is a projective limit of finite groups $\mathtt{G}_i$, then $\mathtt{G}_k=\varprojlim_{i\in I}(\mathtt{G}_i)_k$ is a projective limit of finite constant group schemes. The group $\mathtt{G}_k(k)$ can be identified with $\mathtt{G}$ itself. All prime ideals of $\spec(k[\mathtt{G}_k])$ are maximal (and minimal) and the kernel of some morphism $k[\mathtt{G}_k]\to k$ of $k$-algebras. Thus $\spec(k[\mathtt{G}_k])$ can also be identified with $\mathtt{G}$. Moreover, the topologies on $\spec(k[\mathtt{G}_k])$ and  $\mathtt{G}$ agree. There is a one-to-one correspondence between the closed subgroups of $\mathtt{G}$ and the closed subgroups of $\mathtt{G}_k$. Under this correspondence, the normal open (=closed and of finite index) subgroups of $\mathtt{G}$ correspond to the coalgebraic subgroups of $\mathtt{G}_k$.

\subsection{Formations of algebraic groups} \label{subsec: Formations}

To have a well-behaved notion of pro-$\CC$-groups, the class $\CC$ of algebraic groups needs to satisfy certain closure properties that we will now discuss.

We always assume that $\CC$ is a non-empty class of algebraic groups over $k$ such that any algebraic group isomorphic to an algebraic group in $\CC$ also belongs to $\CC$. A group that belongs to $\CC$ will also be called a $\CC$-group. We call $\CC$ a \emph{formation} if it satisfies the following two conditions.

\begin{enumerate}
    \item $\CC$ is closed under taking quotients, i.e., if $G\twoheadrightarrow H$ is an epimorphism of algebraic groups and $G\in\CC$, then $H\in\CC$.
    \item $\CC$ is closed under subdirect products, i.e., if $H$ is a subdirect product of $G_1,G_2\in\CC$, then $H\in\CC$. (Recall that a closed subgroup $H$ of $G_1\times G_2$ is a subdirect product of $G_1$ and $G_2$ if the projections $H\to G_i$ are epimorphisms for $i=1,2$.)
\end{enumerate}

\begin{rem} \label{rem:equivalent conditions for closed under subdirect products}
    Condition (ii) above has two equivalent reformulations:
    \begin{enumerate}
        \item[\rm{(ii)'}] If $G$ is a proalgebraic group and $N_1,N_2\unlhd G$ are normal closed subgroups such that $G/N_1$ and $G/N_2$ are in $\CC$, then also $G/(N_1\cap N_2)$ is in $\CC$.
        \item[\rm{(ii)''}] If $G$ is an algebraic group and $N_1,N_2\unlhd G$ are normal closed subgroups such that $N_1\cap N_2=1$ and $G/N_1,G/N_2\in\CC$, then $G\in\CC$.
    \end{enumerate}
To avoid trivialities we also assume that a formation contains a non-trivial algebraic group.
\end{rem}

If $\mathtt{C}$ is a formation of (abstract) finite groups (\cite[p. 20]{RibesZalesskii:ProfiniteGroups}), i.e., $\mathtt{C}$ is closed under taking quotients and subdirect products, we let $\CC=\mathtt{C}_k$ denote the class of all algebraic groups isomorphic to $\mathtt{G}_k$ for some $\mathtt{G}$ in $\mathtt{C}$. Then $\CC$ is a formation of algebraic groups.

As illustrated by the following example, most of the familiar classes of algebraic groups are formations.
\begin{ex} \label{ex:formations}
        Note that if $\CC$ is closed under taking quotients, direct products and closed subgroups, then $\CC$ is a formation. Therefore the following classes are formations:
    \begin{enumerate}
        \item The class of all algebraic groups,
        \item the class of all abelian algebraic groups,
        \item the class of all nilpotent algebraic groups,
        \item the class of all solvable algebraic groups,
        \item the class of all unipotent algebraic groups,
        \item the class of all \'{e}tale algebraic groups,
        \item the class of all diagonalizable algebraic groups,
        \item the class of all infinitesimal algebraic groups.
    \end{enumerate}

    Moreover, over a field of charcteristic zero,
    \begin{enumerate}  \setcounter{enumi}{7}
            \item the class of all semisimple algebraic groups and
        \item the class of all reductive linear algebraic groups
    \end{enumerate}
are formations. To be precise, here semisimple and reductive groups are required to be smooth but not necessarily connected. Quotients of semisimple (respectively reductive) algebraic groups are semisimple (respectively reductive). See e.g., \cite[Lemma 19.14]{Milne:AlgebraicGroupsTheTheoryOfGroupSchemesOfFiniteTypeOverAField}. Let us show that in characteristic zero the class of semisimple linear algebraic groups is closed under subdirect products. (The argument for reductive algebraic groups is similar.) Let $G$ be an algebraic group and $N_1,N_2\unlhd G$ closed normal subgroups with $N_1\cap N_2=1$ and $G/N_1,G/N_2$ semi-simple. Let $R(G_{\overline{k}})$ denote the radical of $G_{\overline{k}}$. Since $R(G_{\overline{k}})$ maps into  $R((G/N_1)_{\overline{k}})=1$, we see that $R(G_{\overline{k}})$ is contained in ${(N_1)}_{\overline{k}}$. Similarly, $R(G_{\overline{k}})\subseteq {(N_2)}_{\overline{k}}$. Thus $R(G_{\overline{k}})=1$. Since $G$ is automatically smooth in characteristic zero, it follows that $G$ is semisimple.

The following example shows that over a field of positive characteristic $p$, the classes of all smooth algebraic groups, all semisimple algebraic groups and all reductive algebraic groups are not formations: The subdirect product $H=\{(g_1,g_2)\in\Sl_2\times\Sl_2|\ \operatorname{Fr}_p(g_1)=\operatorname{Fr}_p(g_2)\}$ of $G_1=\Sl_2$ and $G_2=\Sl_2$ is not smooth. Here, for any $2\times 2$ matrix
$$\operatorname{Fr}_p\left(\begin{bmatrix} a & b \\ c & d\end{bmatrix}\right)=\begin{bmatrix} a^p & b^p \\ c^p & d^p\end{bmatrix}$$

The class of connected algebraic groups is also not a formation as shown by the following: the subdirect product $\{(g_1,g_2)\in \Gm^2|\
g_1^2=g_2^4\}$ of $G_1=G_2=\Gm$ is not connected.
\end{ex}

Let $\CC$ be a class of algebraic groups. A proalgebraic group $G$ is called a \emph{pro-$\CC$-group} if $G=\varprojlim G_i$ is a projective limit of $\CC$-groups $G_i$, where the transition maps $G_j\to G_i$ ($j\geq i$) are epimorphisms. The following lemma shows that, when $\CC$ is a formation, every pro-$\CC$-group can canonically be written as a projective limit of $\CC$-groups.

\begin{lemma}
    Let $\CC$ be a formation of algebraic groups and $G$ a pro-$\CC$-group. Then
    \begin{enumerate}
        \item every algebraic quotient of $G$ belongs to $\CC$,
        \item the canonical map $G\to \varprojlim G/N$ is an isomorphism, where the projective limit is taken over all coalgebraic subgroups $N$ of $G$ and
        \item every quotient of a pro-$\CC$-group is a pro-$\CC$-group.
%       \item If moreover $\CC$ is closed under taking closed subgroups, then every closed subgroup of a pro-$\CC$-group is a pro-$\CC$-group.
\end{enumerate}
\end{lemma}
\begin{proof}
    Let $G=\varprojlim G_i$ be a projective limit of $\CC$-groups $G_i$, with surjective transition maps $G_j\to G_i$. Then $k[G]$ is the directed union of the $k[G_i]$'s. Let $H$ be an algebraic quotient of $G$. Since $k[H]\subseteq k[G]$ is finitely generated, $k[H]$ is contained in some $k[G_i]$ and therefore $H$ is a quotient of $G_i$. As $\CC$ is closed under taking quotients, $H$ is a $\CC$-group.

    Let us prove (ii): The coalgebraic subgroups of $G$ form a directed set under $N\leq N'$ if $N\supseteq N'$ and the transition maps $G/N'\to G/N$ yield a projective system. Since $k[G]$ is the directed union of its finitely generated Hopf subalgebras (\cite[Section 3.3]{Waterhouse:IntroductiontoAffineGroupSchemes}), we see that $G= \varprojlim G/N$.

    For a quotient $H$ of $G$ we have $k[H]\subseteq k[G]$. The proalgebraic group $H$ is the projective limit of its algebraic quotients $H_i$ and $k[H_i]\subseteq k[H]\subseteq k[G]$. By (i) the $H_i$'s are $\CC$-groups and thus $H$ is a pro-$\CC$-group.
\end{proof}
The following lemma will be useful for constructing free pro-$\CC$-groups.

\begin{lemma} \label{lemma: maximal pro C quotient}
    Let $G$ be a proalgebraic group and $\CC$ a formation of algebraic groups. Then $G$ has a maximal pro-$\CC$-quotient, i.e., there exists a pro-$\CC$-group $\pi_\CC(G)$ together with an epimorphism $G\twoheadrightarrow \pi_\CC(G)$ such that for any morphism $\f\colon G\to H$ of proalgebraic groups with $\f(G)$ a pro-$\CC$-group, there exists a unique morphism $\pi_\CC(G)\to H$ such that
    $$
    \xymatrix{
G \ar@{->>}[rr] \ar[rd] & &     \pi_\CC(G) \ar@{..>}[ld] \\
 & H &
}
    $$
    commutes.
\end{lemma}
\begin{proof}
Let $I$ denote the set of all normal closed subgroups $N$ of $G$ such that $G/N\in\CC$. We define a partial order on $I$ by $N_1\leq N_2$ if $N_1\supseteq N_2$. Because $\CC$ is closed under subdirect products (Remark~ \ref{rem:equivalent conditions for closed under subdirect products}) we know that $N_1\cap N_2\in \CC$ if $N_1,N_2\in\CC$. Thus $I$ is directed.
For $N_1\leq N_2$ we have a canonical epimorphism $G/N_2\twoheadrightarrow G/N_1$ and the collection of all these maps is a projective system. Let $\pi_\CC(G)$ denote the corresponding projective limit. Then $\pi_\CC(G)$ is a pro-$\CC$-group and $k[\pi_\CC(G)]\subseteq k[G]$ is the union of all finitely generated Hopf subalgebras $k[G_i]$ of $G$ such that $G_i$ is a $\CC$-group.

If $\f\colon G\to H$ is a morphism with $\f(G)$ a pro-$\CC$-group, then $k[\f(G)]\subseteq k[G]$ and since $k[\f(G)]$ is a union of $k[H_i]$'s with $H_i\in\CC$, we see that $k[\f(G)]\subseteq k[\pi_\CC(G)]$. So $\f$ factors uniquely through $G\twoheadrightarrow \pi_\CC(G)$.
\end{proof}

\begin{ex}
    If $G$ is an algebraic group and $\CC$ the class of all \'{e}tale algebraic groups, then $\pi_\CC(G)$ agrees with $\pi_0(G)=G/G^o$, the group of connected components of $G$ (\cite[Def. 2.38]{Milne:AlgebraicGroupsTheTheoryOfGroupSchemesOfFiniteTypeOverAField}).
\end{ex}

\subsection{Coalgebraic subgroups and convergence to one}

%Let $X$ be a subset of $G(\kb)$. As the intersection of closed subgroups $H$ of $G$ with $X\subset H(\kb)$ also has this property, there exists a smallest closed subgroup $H$ of $G$ such that $X\subset H(\kb)$. We call it the \emph{closed subgroup generated by $X$}. If $H=G$ we also say that \emph{$X$ generates $G$}.

%
%
%A subset $X$ of $G(\kb)$ \emph{converges to $1$} if $X\smallsetminus N(\kb)$ is finite for every coalgebraic subgroup $N$ of $G$.

Let $G=\varprojlim G/N$ be a proalgebraic group, written as the projective limit of the canonical projective system. Then also $G(\overline{k})=\varprojlim (G/N)(\overline{k})$ and it is possible to topologize $G(\overline{k})$ with the limit topology, where the $(G/N)(\overline{k})$ are considered with the discrete topology (rather than the Zariski topology). Then $\{N(\overline{k})|\ N\leq G \text{ is coalgebraic} \}$ is a neighborhood basis at $1$ for $G(\kb)$. We have no need to actually use this topology on $G(\overline{k})$, but we will introduce some topological language which is reminiscent of this idea.

\begin{defi}
    Let $G$ be a proalgebraic group.  A set $\N$ of coalgebraic subgroups of $G$ is a \emph{neighborhood basis} at $1$ for $G$, if
    for every coalgebraic subgroup $N$ of $G$, there exists $N'\in \N$ with $N'\leq N$.
\end{defi}
Thus a set $\N$ of coalgebraic subgroups of $G$ is a neighborhood basis at $1$ for $G$ if and only if \begin{enumerate}
    \item $\N$ is downward directed, i.e., for $N_1,N_2\in \N$ there exists $N_3\in\N$ with $N_3\leq N_1, N_2$.
    \item $\bigcap_{N\in\N}N=1$.
\end{enumerate}

From a neighborhood basis $\N$ at $1$ for $G$ we obtain a projective system $G/N\twoheadrightarrow G/N'$ ($N,N'\in\N$ and $N\leq N'$) with $G=\varprojlim_{N\in\N} G/N$.
Conversely, if $G=\varprojlim_{i\in I}G_i$ is a projective limit of algebraic groups, then the set of all kernels of the morphisms $G\to G_i$ is a neighborhood basis at $1$ for $G$.

\begin{lemma} \label{lemma: co-algebrai and subgroup}
    Let $G$ be a proalgebraic group, $H\leq G$ a closed subgroup and $N\leq H$ a coalgebraic subgroup of $H$. Then there exists a coalgebraic subgroup $N'$ of $G$ such that $H\cap N'\leq N$.
\end{lemma}
\begin{proof}
    Since $k[G]$ is the directed union of its finitely generated Hopf subalgebras and $k[H/N]\subseteq k[H]$ is finitely generated, it follows that there exists a finitely generated Hopf subalgebra of $k[G]$ whose image in $k[H]$ contains $k[H/N]$. This Hopf subalgebra is necessarily of the form $k[G/N']$ for some coalgebraic subgroup $N'$ of $G$. The image of $k[G/N']$ in $k[H]$ is $k[H/H\cap N']$. Since $k[H/N]\subseteq k[H/H\cap N']$, it follows that $H\cap N'\subseteq N$.
\end{proof}
From Lemma \ref{lemma: co-algebrai and subgroup} we immediately obtain:
\begin{cor} \label{cor: neighborhood basis for subgroup}
    Let $G$ be a proalgebraic group and $H\leq G$ a closed subgroup. If
    $\N$ is a neighborhood basis at $1$ for $G$, then $\N(H)=\{N\cap H|\ N\in \N\}$ is a neighborhood basis at $1$ for $H$.
\end{cor}

\begin{cor} \label{cor: intersection 1}
    Let $G$ be a proalgebraic group and $H\leq G$ a closed subgroup that is algebraic. Then there exists a coalgebraic subgroup $N'$ of $G$ with $N'\cap H=1$.
\end{cor}
\begin{proof}
    As $H$ is algebraic we can chose $N=1$ in Lemma \ref{lemma: co-algebrai and subgroup}.
\end{proof}

\begin{cor} \label{cor: subgroup is pro C}
    Let $\CC$ be a formation of algebraic groups and $G$ a pro-$\CC$-group. If $\CC$ is closed under taking closed subgroups, then every closed subgroup of $G$ is a pro-$\CC$-group.
\end{cor}
\begin{proof}
    Let $H$ be a closed subgroup of $G$ and let $\N$ be a neighborhood basis at $1$ for $G$. According to Corollary \ref{cor: neighborhood basis for subgroup} the set $\{H\cap N|\ N\in\N\}$ is a neighborhood basis at $1$ for $H$. Thus $H\simeq\varprojlim H/(H\cap N)$. But $H/(H\cap N)$ injects into $G/N$ and is therefore a $\CC$-group.
\end{proof}

The free proalgebraic groups that we want to consider are free on a set $X$ such that the map from $X$ to the proalgebraic group has a certain property which is explained in the following definition.

\begin{defi}
    Let $X$ be a set, $G$ a proalgebraic group and $R$ a $k$-algebra. A map $\varphi:X\to G(R)$ \emph{converges to $1$} if for every coalgebraic subgroup $N$ of $G$ almost all elements of $X$ map into $N(R)$, i.e., $X\smallsetminus\varphi^{-1}(N(R))$ is finite for every coalgebraic subgroup $N$ of $G$.
\end{defi}
For $X\subseteq G(R)$, we say that \emph{$X$ converges to $1$} if the inclusion map converges to $1$. If $G$ is algebraic, it follows that $\varphi$ converges to $1$ if and only if almost all elements of $X$ map to $1$.

\begin{lemma} \label{lemma: morphism and convergence to 1}
    Let $\f\colon G\to H$ be a morphism of proalgebraic groups, $X$ a set, $R$ a $k$-algebra and $\varphi\colon X\to G(R)$ a map converging to $1$. Then $\f_R\circ\varphi\colon X\to H(R)$ converges to $1$.
\end{lemma}
\begin{proof}
    If $N$ is a coalgebraic subgroup of $H$, then $\f^{-1}(N)$ is a coalgebraic subgroup of $G$.
\end{proof}

\begin{lemma} \label{lemma: converges to one}
    Let $X$ be a set, $G$ a proalgebraic group, $H\leq G$ a closed subgroup and $R$ a $k$-algebra. If the map $\varphi\colon X\to G(R)$ converges to $1$ and $\varphi(X)\subseteq H(R)$, then $\varphi\colon X\to H(R)$ converges to $1$.
\end{lemma}
\begin{proof}
    Let $N$ be a coalgebraic subgroup of $H$. By Lemma \ref{lemma: co-algebrai and subgroup} there exists a coalgebraic subgroup $N'$ of $G$ with $H\cap N'\leq N$. The kernel of $H(R)\to (G/N')(R)$ is $(H\cap N')(R)$ and almost all elements of $X$ map to $1$ in $(G/N')(R)$. Therefore almost all elements of $X$ map into $(H\cap N')(R)\subseteq N(R)$.
\end{proof}

\subsection{Generating proalgebraic groups}

Let $G$ be a proalgebraic group, $R$ a $k$-algebra and $X\subseteq G(R)$. As the intersection of closed subgroups $H$ of $G$ with $X\subseteq H(R)$ also has this property, there exists a smallest closed subgroup $\langle X\rangle$ of $G$ such that $X\subseteq\langle X\rangle(R)$. We call it the \emph{closed subgroup generated by $X$}. If $\langle X\rangle=G$, we also say that \emph{$X$ generates $G$} or that $X$ is a generating set of $G$.

\begin{lemma} \label{lemma: generates reduced subgroup}
    Let $G$ be a proalgebraic group, $R$ a $k$-algebra and $X\subseteq G(R)$. If $k$ is perfect and $R$ reduced, then $\langle X\rangle$ is reduced (and therefore geometrically reduced).
\end{lemma}
\begin{proof}
    Since $k$ is perfect, the underlying reduced subscheme $G_{\operatorname{red}}$ of a proalgebraic group $G$ is a closed subgroup of $G$. (See \cite[Cor. 1.39]{Milne:AlgebraicGroupsTheTheoryOfGroupSchemesOfFiniteTypeOverAField} for the algebraic case or \cite[Expos\'{e} $\operatorname{VI_{A}}$, Section 0.2]{Grothendieck:SGA3_1} for the general case.)
    So  $\langle X\rangle_{\operatorname{red}}$ is a closed subgroup of $G$. It therefore suffices to show that $X\subseteq \langle X\rangle_{\operatorname{red}}(R)$.

    If $H$ is a closed subgroup of $G$, an $x\in X\subseteq G(R)=\Hom(k[G],R)$ lies in $H(R)$ if and only if the defining ideal $\I(H)\subseteq k[G]$ is contained in $\ker(x)$. By definition $X\subseteq \langle X\rangle(R)$, so $\I(\langle X\rangle)\subseteq\ker(x)$ for all $x\in X$. Because $R$ is reduced, $\ker(x)$ is a radical ideal and therefore
    $\I(\langle X\rangle_{\operatorname{red}})=\sqrt{\I(\langle X\rangle)}\subseteq\ker(x)$ for all $x\in X$. Thus $X\subseteq \langle X\rangle_{\operatorname{red}}(R)$.
\end{proof}

If $R$ is not reduced, $\langle X\rangle$ need not be reduced, even when $k$ is perfect. For example, if $k$ is a field of characteristic $2$, $G=\Ga$, $R=k[\epsilon]=k[t]/(t^2)$ and $X=\{\epsilon\}$, then $\langle X\rangle=\alpha_2=\{g\in\Ga|\ g^2=0\}$.
The following example shows that $\langle X\rangle$ need not be reduced if $k$ is not perfect, even when $R=\overline{k}$.

\begin{ex} \label{ex: Zariski closure non reduced}
    Let $k$ be a field of characteristic $2$ and $a\in k$ such that $a^{\frac{1}{2}}\in\overline{k}\smallsetminus k$. Let $G=\Ga$ and $X=\{0,a^\frac{1}{2}\}\leq G(\overline{k})$. We will show that  $\langle X\rangle=H=\{g\in\Ga|\ g^2(g^2-a)=0\}$.  Clearly $H$ is a closed subgroup of $G$ with $X\subseteq H(\overline{k})$. There are only four subschemes of $H$, their defining equations are \begin{itemize}
        \item $x=0$,
        \item $x^2=0$,
        \item $x^2-a=0$ or
        \item $x(x^2-a)=0$.
    \end{itemize}
    These are either not subgroups or do not contain $X$. Thus $\langle X\rangle=H$.

\end{ex}

\begin{lemma} \label{lemma: morphisms and generation}
    Let $\f\colon G\to H$ be a morphism of proalgebraic groups, $R$ a $k$-algebra and $X\subseteq G(R)$. Then $\f(\langle X\rangle)=\langle\f_R(X)\rangle$. In particular, if $\langle X\rangle=G$ and $\f$ is an epimorphism, then $\langle\f_R(X)\rangle=H$.
\end{lemma}
\begin{proof}
    As $\f_R(X)\subseteq \f(\langle X\rangle)(R)$, clearly, $\langle\f_R(X)\rangle \subseteq \f(\langle X\rangle)$. On the other hand, $\f^{-1}(\langle\f_R(X)\rangle)$ is a closed subgroup of $G$ whose $R$-points contain $X$. Thus $\langle X\rangle\subseteq  \f^{-1}(\langle\f_R(X)\rangle)$ and therefore $\f(\langle X\rangle)\subseteq\langle\f_R(X)\rangle$.
    %
    %
    %    $X\subseteq \f^{-1}(\langle\f_R(X)\rangle)(R)$.
\end{proof}

%\begin{lemma} \label{lemma: generators and morphisms}
%   Let $\f\colon G\to H$ be a morphism of proalgebraic groups and $X\subseteq G(k)$. Then $\f(\langle X\rangle)=\langle\f_k(X)\rangle$.
%\end{lemma}
%\begin{proof}
%   As $\f_k(X)\subseteq\f(\langle X\rangle)(k)$, clearly $\langle\f_k(X)\rangle\subseteq\f(\langle X\rangle)$. On the other hand, $\f^{-1}( \langle\f_k(X)\rangle)$ is a closed subgroup of $G$ with $X\subseteq \f^{-1}( \langle\f_k(X)\rangle)(k)$. So $\langle X\rangle \subseteq \f^{-1}( \langle\f_k(X)\rangle)$ and $\f(\langle X\rangle)\subseteq\langle\f_k(X)\rangle$.
%\end{proof}
%
%\begin{lemma} \label{lemma: generation coalgebraic}
%   Let $G$ be a proalgebraic group and $X\subseteq G(k)$. Then $X$ generates $G$ if and only the image of $X$ in $(G/N)(k)$ generates $G/N$ for every coalgebraic subgroup $N$ of $G$.
%\end{lemma}
%\begin{proof}
%   It follows from Lemma \ref{lemma: generators and morphisms} that the image of $X$ in $(G/N)(k)$ generates $G/N$ if $X$ generates $G$.
%
%   Conversely, let $H=\langle X\rangle$. By assumption and  Lemma \ref{lemma: generators and morphisms}, it follows that $H\to G/N$ is surjective for every coalgebraic subgroup $N$ of $G$. In terms of Hopf-algebras, this means that for every Hopf subalgebra $k[G/N]$ of $k[G]$ that is finitely generated as a $k$-algebra, the induced map $k[G/N]\to k[H]$ is injective. However, $k[G]$ is the union of all its finitely generated Hopf-subalgebras. Thus $k[G]\to k[H]$ is injective and therefore $H=G$.
%\end{proof}

\subsection{Existence of free pro-$\CC$-groups}

We are now prepared to introduce free proalgebraic groups.

\begin{theo} \label{theo: free pro-alg group}
    Let $X$ be a set, $\CC$ a formation and $R$ a $k$-algebra. There exists a pro-$\CC$-group $\Gamma=\Gamma_{R/k}^{\CC}(X)$ and a map $\iota\colon X\to \Gamma(R)$ such that
    \begin{enumerate}
        \item   $\iota$ converges to $1$ and $\langle \iota(X)\rangle$ is a pro-$\CC$-group,
        \item the pair $(\Gamma, \iota)$ is universal, i.e., for every pro-$\CC$-group $G$ and every map $\varphi\colon X\to G(R)$ converging to $1$ with $\langle \varphi(X)\rangle$ a pro-$\CC$-group, there exists a unique morphism $\psi\colon \Gamma\to G$ of proalgebraic groups such that
        \begin{equation} \label{eqn: universal prop}
        \xymatrix{
            X \ar^{\iota}[rr] \ar_{\varphi}[rd] &  & \Gamma(R) \ar^{\psi_R}[ld] \\
            & G(R) &
        }
        \end{equation}
        commutes.
    \end{enumerate}
Moreover, we have $\langle\iota(X)\rangle=\Gamma$.
\end{theo}
\begin{proof}
    We first show that $\langle \iota(X)\rangle =\Gamma$ for any pair $(\Gamma,\iota)$ satisfying (i) and (ii). Let $G=\langle \iota(X)\rangle\leq \Gamma$ and let $\varphi\colon X\to G(R)$ be such that $\varphi(x)=\iota(x)$ for all $x\in X$. Since $\varphi$ converges to $1$ by Lemma~\ref{lemma: converges to one}, it follows from (ii) that there exists a morphism $\psi\colon \Gamma\to G$ such that (\ref{eqn: universal prop}) commutes.
    Then $\f\colon \Gamma\xrightarrow{\psi} G\to \Gamma$ is such that
        $$
    \xymatrix{
        X \ar^{\iota}[rr] \ar_{\iota}[rd] &  & \Gamma(R) \ar^{\f_R}[ld] \\
        & \Gamma(R) &
    }
    $$
    commutes. As also $\id\colon\Gamma\to \Gamma$ has this property, it follows from the uniqueness in (ii) that $\f=\id$. Therefore $G=\Gamma$ as desired.

    We will next construct $\Gamma_{R/k}^{\CC}(X)$ for the formation $\CC$ of all algebraic groups.
    Let $F_X$ be the free (abstract) group on $X$. A \emph{cofinite representation} of $F_X$ (over $R$) is a pair $(V,\f)$, where $V$ is a finite dimensional $k$-vector space and $\f\colon F_X\to \Gl(V\otimes_k R)=\Gl(V)(R)$ is a morphism of groups such that almost all $x\in X$ map to the identity of $\Gl(V)(R)$.
    We may write $V$ instead of $(V,\f)$ and $g(a)$ instead of $\f(g)(a)$ for $g\in F_X$ and $a\in V\otimes_k R$.

    A morphism $(V,\f)\to (V',\f')$ of cofinite representations is a $k$-linear map $\eta\colon V\to V'$ such that we have
    $(\eta\otimes R)(\f(g)(a))=\f'(g)((\eta\otimes R)(a))$ for every $g\in F_X$ and $a\in V\otimes_k R$. Clearly the cofinite representations of $F_X$ are a category. We will show that they indeed are a neutral  Tannakian category (\cite{DeligneMilne:TannakianCategories, Deligne:categoriestannakien}) in a natural way.

    For two cofinite representations $(V,\f)$ and $(V',\f')$, the tensor product $V\otimes_k V'$ becomes a cofinite representation by $$g(a\otimes b)=g(a)\otimes g(b)$$ for $g\in F_X$ and $a\otimes b\in (V\otimes_k R)\otimes_R (V'\otimes_k R)=(V\otimes_kV')\otimes_k R$.
    For three cofinite representations $V,V',V''$, the map $(V\otimes_k V')\otimes_k V''\to V\otimes_k(V'\otimes_k V'')$ is a morphism of cofinite representations. Therefore the usual associativity constraint on $k$-vector spaces induces an associativity constraint on the category of cofinite representations. Similarly, we obtain a compatible commutativity constraint. Thus the cofinite representations are a symmetric monoidal category. The identity object $\mathds{1}$ is given by the one-dimensional $k$-vector space $k$, with $\f\colon F_X\to \Gl(k\otimes_k R)=\Gl_1(R)$ the trivial map.

    % and the isomorphism $V\to V\otimes_k V$ determined by $e\mapsto e\otimes e$.
    For a cofinite representation $(V,\f)$ the dual vector space $V^\vee=\Hom_k(V,k)$ is a cofinite representation via $g(f)(a)=f(g^{-1}(a))$ for $g\in F_X$, $f\in V^\vee\otimes_k R=\Hom_R(V\otimes_k R,R)$ and $a\in V\otimes_k R$. The maps $V\otimes_k V^\vee\to \mathds{1}$ and $\mathds{1}\to V^\vee\otimes_k V$ are morphisms of cofinite representations. Therefore our category is rigid. (Cf. \cite[2.1.2]{Deligne:categoriestannakien}.)

    If $\eta\colon V\to V'$ is a morphism of cofinite representations, the kernel $\ker(\eta)$ is naturally a cofinite representation because  $\ker(\eta)\otimes_k R=\ker(\eta\otimes R)$.
     If $V$ is a cofinite representation and $W\leq V$ a subspace such that $W\otimes_k R\subseteq V\otimes_k R$ is stable under $F_X$, then there is an induced action on $(W\otimes_k R)/(V\otimes_k R)=(W/V)\otimes_k R$ and $W/V$ is a cofinite representation. So we see that if $\eta\colon V\to V'$ is a morphism of cofinite representations, then $\eta(V)\otimes_k R$ is stable under $F_X$ and $V'\to V'/\eta(V)$ is a cokernel of $\eta$. So the cofinite representations of $F_X$ are an abelian category.

     As $k$ can be identified with the endomorphisms of $\mathds{1}$, we conclude that the cofinite representations of $F_X$ are a tensor category (in the sense of \cite[2.1]{Deligne:categoriestannakien}).

     To conclude the description of our neutral tannakian category we need a fibre functor. We take it to be the forgetful functor $\omega$, that associates the underlying vector space $V$ to a cofinite representations $(V,\f)$.

     Let $\Gamma=\underline{\Aut}^\otimes(\omega)$ be the fundamental group of this tannakian category. An element of $\Gamma(R)$ is, by definition, a family $(\lambda_V)$, indexed by all cofinite representations, where $\lambda_V\colon V\otimes_kR\to V\otimes_k R$ is an $R$-linear automorphism such that
     $$\xymatrix{
        V\otimes_k R \ar_{\eta\otimes R}[d] \ar^{\lambda_V}[r] & V\otimes_k R \ar^{\eta\otimes R}[d] \\
        V'\otimes_k R \ar^{\lambda_{V'}}[r] & V'\otimes_k R
     }
      $$
      commutes for every morphism $\eta\colon V\to V'$ of cofinite representations, $\lambda_{V\otimes_k V'}=\lambda_{V}\otimes\lambda_{V'}$ for all cofinite representations $V$, $V'$ and $\lambda_{\mathds{1}}$ is the identity map (on $R$).

     We have a group homomorphism $F_X\to\Gamma(R)$, given by sending $g\in F_X$ to
    $(\f(g))_{(V,\f)}$, where $(V,\f)$ ranges over all cofinite representations of $F_X$.
    We note that $\Gamma$ is the projective limit of the algebraic groups $\underline{\Aut}^\otimes(\omega|_{\langle (V,\f)\rangle_\otimes})$, where $\langle (V,\f)\rangle_\otimes$ denotes the tensor category generated by the cofinite representation $(V,\f)$, and that $\underline{\Aut}^\otimes(\omega|_{\langle (V,\f)\rangle_{\otimes}})$ naturally embeds into $\Gl(V)$ as a closed subgroup. The category of cofinite representations of $F_X$ is equivalent (via $\omega$) to the category of (finite dimensional) representations of $\Gamma$. Explicitly, a cofinite representation $(V,\f)$ corresponds to the canonical projection
    $$\Gamma\to \underline{\Aut}^\otimes(\omega|_{\langle (V,\f)\rangle_\otimes})\subseteq\Gl(V).$$
    In particular, the universal element $\id\colon\Gamma\to\Gamma$ corresponds to the group homomorphism $F_X\to\Gamma(R)$.

    Let $\iota\colon X\to \Gamma(R)$ be the map induced by restricting $F_X\to\Gamma(R)$ to $X$. We have to show that $\iota$ converges to $1$. So let $N\unlhd \Gamma$ be a coalgebraic subgroup. Since the kernel of $\Gamma(R)\to(\Gamma/N)(R)$ is $N(R)$, it suffices to show that almost all elements of $X$ map to $1$ under $F_X\to\Gamma(R)\to (\Gamma/N)(R)$.

 Let $(V,\f)$ be a cofinite representation of $F_X$ that induces a faithful representation of $\Gamma/N$. Then the diagram
    $$
    \xymatrix{
        F_X \ar^-\f[r]  \ar[d] & \Gl(V)(R) \\
        \Gamma(R) \ar[r] & (\Gamma/N)(R) \ar[u]
    }
    $$
    commutes.
    Because $(\Gamma/N)(R)\to \Gl(V)(R)$ is injective and $\f$ is cofinite, it follows that almost all elements of $X$ map into $N(R)$. Thus $\iota$ converges to $1$.

    We next establish the universal property:
%    similar to the one stated in the theorem (but even without the assumption that $\langle\varphi(X)\rangle=G$). Indeed,
 We claim that for every proalgebraic group $G$ and every map $\varphi\colon X\to G(R)$ converging to $1$, there exists a unique morphism $\psi\colon\Gamma\to G$ such that (\ref{eqn: universal prop}) commutes.
%   To verify the universal property of $\Gamma$ let $G$ be a proalgebraic group and $\varphi\colon X\to G(R)$ a map that converges to $1$.

    Note that $\varphi$ induces a morphism of groups $F_X\to G(R)$.
    Because $\varphi$ converges to $1$, every (finite dimensional, $k$-linear) representation of $G$ induces a cofinite representation of $F_X$.
% the representation of $F_X$ induced by a representation of $G$ is cofinite.
We thus have a tensor functor $F\colon\Rep(G)\to \Rep_{cf}(F_X)=\Rep(\Gamma)$, compatible with the fibre functors, where
    $\Rep(G)$ denotes the tannakian category of representations of $G$ and $\Rep_{cf}(F_X)$ denotes the category of cofinite representations of $F_X$. By \cite[Cor. 2.9]{DeligneMilne:TannakianCategories}, the tensor functor $F$ is induced by a unique morphism $\psi\colon\Gamma\to G$ of proalgebraic groups. By construction, the diagram
    $$
    \xymatrix{
        \Gamma(R) \ar^{\psi_R}[dr] & \\
        F_X \ar[r] \ar[u] & G(R) \\
        X   \ar_{\varphi}[ru] \ar[u]
    }
    $$
    commutes. Moreover, if $\psi'\colon\Gamma\to G$ is another morphism of proalgebraic groups such that
    $$
    \xymatrix{
        \Gamma(R) \ar^{\psi'_R}[dr] & \\
        F_X \ar[r] \ar[u] & G(R) \\
        X   \ar_{\varphi}[ru] \ar[u]
    }
    $$
    commutes, then $\psi$ and $\psi'$ induce the same tensor functor $\Rep(G)\to \Rep_{cf}(F_X)$ and therefore $\psi=\psi'$ (loc. cit.). %This establishes the universal property of $(\Gamma,\iota)$.
     So the existence of $\Gamma=\Gamma^\CC_{R/k}(X)$, is established when $\CC$ is the formation of all algebraic groups.
%
%
%    We will next show that $\langle\iota(X)\rangle=\Gamma$.
%
%   By Lemma \ref{lemma: converges to one} the map $\iota\colon X\to \langle\iota(X)\rangle(R)$ converges to $1$. So, by the universal property just established, there exists a unique morphism $\psi\colon\Gamma\to \langle\iota(X)\rangle$ such that
%       $$
%   \xymatrix{
%       X \ar^{\iota}[rr] \ar_{\iota}[rd] &  & \Gamma(R) \ar^{\psi_R}[ld] \\
%       & \langle\iota(X)\rangle(R) &
%   }
%   $$
%   commutes. Let $i\colon \langle\iota(X)\rangle\to \Gamma$ denote the inclusion. Then
%       $$
%   \xymatrix{
%       X \ar^{\iota}[rr] \ar_{\iota}[rd] &  & \Gamma(R) \ar^{(i\circ\psi)_R}[ld] \\
%       & \Gamma(R) &
%   }
%   $$
%   commutes and thus $i\circ\psi$ is the identity on $\Gamma$. This implies that $\langle\iota(X)\rangle=\Gamma$.

    Let us now assume that $\CC$ is an arbitrary formation and let $(\Gamma,\iota)$ be as above. We know from Lemma \ref{lemma: maximal pro C quotient} that $\Gamma$ has a maximal pro-$\CC$-quotient $\Gamma_{R/k}^\CC(X)=\pi_\CC(\Gamma)$. Let $\iota_{\CC}$ denote the composition
$$\iota_{\CC}\colon X\xrightarrow{\iota} \Gamma(R)\to \pi_{\CC}(\Gamma)(R).$$
By Lemma \ref{lemma: morphism and convergence to 1} the map $\iota_\CC$ converges to $1$. Since $\langle\iota(X)\rangle =\Gamma$ (as established in the first paragraph) it follows, using Lemma \ref{lemma: morphisms and generation}, that $\langle\iota_\CC(X)\rangle=\pi_{\CC}(\Gamma)$. In particular, $\langle \iota_{\CC}(X)\rangle$ is a pro-$\CC$-group.

    To verify the universal property of $(\pi_{\CC}(\Gamma),\iota_{\CC})$, let $G$ be a pro-$\CC$-group and $\varphi\colon X\to G(R)$ a map converging to $1$ such that $\langle \varphi(X)\rangle$ is a pro-$\CC$-group. By the universal property of $(\Gamma, \iota)$, there exists a unique morphism $\f\colon\Gamma\to G$ such that $\varphi=\f_R\circ\iota$. Using Lemma \ref{lemma: morphisms and generation}, we see that
$$\f(\Gamma)=\f(\langle \iota(X)\rangle)=\langle\f_R(\iota(X))\rangle=\langle\varphi(X)\rangle.$$
Thus $\f(\Gamma)$ is a pro-$\CC$-group and by the universal property of $(\pi_\CC(\Gamma),\pi)$, there exists a unique morphism $\psi\colon\pi_{\CC}(\Gamma)\to G$ such that $\f=\psi\circ\pi$.
%
%       $$
%    \xymatrix{
%       \Gamma \ar@{->>}[rr] \ar_-\f[rd] & &    \pi_\CC(\Gamma) \ar^\psi[ld] \\
%       & G &
%    }
%    $$
%    commutes.
%
The commutativity of
    $$
    \xymatrix{
    X \ar^-{\iota}[r] \ar_\varphi[rd] & \Gamma(R) \ar^-{\f_R}[d] \ar^-{\pi_R}[r] & \pi_{\CC}(\Gamma)(R) \ar^-{\psi_R}[ld] \\
    & G(R) &
    }
    $$
 shows that $\varphi=\psi_R\circ\iota_{\CC}$ as required. If $\psi'\colon \pi_{\CC}(\Gamma)\to G$ is another morphism such that
$$\varphi=\psi'_R\circ\iota_{\CC}=(\psi'\circ\pi)_R\circ\iota,$$
it follows from the universal property of $(\Gamma,\iota)$ that $\psi'\circ\pi=\f$. But then the universal property of $(\pi_{\CC}(\Gamma),\pi)$ yields $\psi'=\psi$.
\end{proof}
%
%It the above proof we have seen the following:
Clearly, the pair $(\Gamma^\CC_{R/k}(X),\iota) $ is uniquely determined (up to isomorphism) by the universal property.

\begin{defi} \label{defi: free proalgebraic group}
    The proalgebraic group $\Gamma^\CC_{R/k}(X)$ from Theorem \ref{theo: free pro-alg group} is called the \emph{free pro-$\CC$-group on $X/R$}. The proalgebraic group
    $\Gamma^\CC(X)=\Gamma^\CC_{\overline{k}/k}(X)$ is called the \emph{free pro-$\CC$-group on $X$} (over $k$).
\end{defi}
%We may also refer to $\Gamma^\CC(X)$ as the free pro-$\CC$-group on $|X|$-generators.

    If $\CC$ is a formation of algebraic groups that is closed under taking subgroups, then the conditions $\langle\iota(X)\rangle$ is a pro-$\CC$-group in (i) and $\langle\varphi(X)\rangle$ is a pro-$\CC$-group in (ii) of Theorem \ref{theo: free pro-alg group} can be dropped (Corollary \ref{cor: subgroup is pro C}).

%\begin{rem} \label{rem: if formation subgroup closed can drop}
%   If $\CC$ is a formation of algebraic groups that is closed under taking subgroups. Then the conditions $\langle\iota(X)\rangle=\Gamma$ in (i) and $\langle\varphi(X)\rangle=G$ in (ii) of Theorem \ref{theo: free pro-alg group} can be dropped.
%\end{rem}
%\begin{proof}
%   Assume that $\Gamma$ is a pro-$\CC$-group, $\iota\colon X\to \Gamma(R)$ converges to $1$ and
%   for every pro-$\CC$-group $G$ with a map $\varphi\colon X\to G(R)$ converging to $1$, there exists a unique $\psi\colon \Gamma\to G$ such that $\varphi=\psi_R\circ\iota$. We have to show that $\langle\iota(X)\rangle=\Gamma$.
%
%   By Corollary \ref{cor: subgroup is pro C} the closed subgroup $\langle\iota(X)\rangle$ of $\Gamma$ is a pro-$\CC$-group and by Lemma \ref{lemma: converges to one} the induced map $\varphi\colon X\to\langle\iota(X)\rangle$ converges to $1$. Now we can argue as in the proof of Theorem \ref{theo: free pro-alg group} that $\langle\iota(X)\rangle=\Gamma$.
%
%   Conversely, if $\varphi\colon X\to G$ is a map from $X$ to a pro-$\CC$-group that converges to $1$, then, as argued above, $\langle \varphi(X)\rangle$ is a pro-$\CC$-group and the induced map $\varphi\colon X\to \langle \varphi(X)\rangle$ converges to $1$. Thus the universal property of $\Gamma^\CC_{R/k}(X)$ applied to this map yields the desired morphism $\psi\colon  \Gamma^\CC_{R/k}(X)\to \langle \varphi(X)\rangle\subseteq G$.
%\end{proof}

\begin{rem} \label{rem: can reduce to algebraic}
    For a pro-$\CC$-group to be the free pro-$\CC$-group on $X/R$ it suffices that it satisfies the mapping property of Theorem \ref{theo: free pro-alg group} for any $\CC$-group $G$.
\end{rem}
\begin{proof}
    Let $\Gamma$ be a pro-$\CC$-group and $\iota\colon X\to \Gamma(R)$ be a map converging to $1$ with $\langle\iota(X)\rangle$ a pro-$\CC$-group such that for every $\CC$-group $G$ with a map $\varphi\colon X\to G(R)$ converging to $1$ with $\langle \varphi(X)\rangle$ a $\CC$-group, there exists a unique morphism $\psi\colon \Gamma\to G$ with $\varphi=\psi_R\circ\iota$.

    Let $G$ be a pro-$\CC$-group and $\varphi\colon X\to G(R)$ a map converging to $1$ with $\langle \varphi(X)\rangle$ a pro-$\CC$-group. We may write $G=\varprojlim_{i\in I} G_i$ as a projective limit of $\CC$-groups $G_i$ with surjective transition maps $G_i\twoheadrightarrow G_j$ ($i\geq j$). Then the induced maps $\varphi_i\colon X\xrightarrow{\iota} G(R)\xrightarrow{(\pi_i)_R}G_i(R)$ also converge to $1$ (Lemma~\ref{lemma: morphism and convergence to 1}) and by Lemma \ref{lemma: morphisms and generation} $\langle \varphi_i(X)\rangle=\pi_i(\langle\varphi(X)\rangle)$  is a $\CC$-group. Thus there exist unique morphisms $\psi_i\colon\Gamma\to G_i$ with $\varphi=(\psi_i)_R\circ \iota$. By uniqueness the $\psi_i$ must be compatible and therefore define a morphism $\psi\colon \Gamma\to G$. This is the unique $\psi$ such that $\varphi=\psi_R\circ\iota$.
\end{proof}

\begin{rem} \label{rem: free group is reduced}
If $k$ is perfect, $\Gamma^\CC(X)$ is geometrically reduced by Lemma \ref{lemma: generates reduced subgroup}. If $k$ is not perfect, $\Gamma^\CC(X)$ need not be reduced. For example, if $\CC$ is the formation of all algebraic groups over a field of characteristic $2$, the non-reduced group $H$ from Example \ref{ex: Zariski closure non reduced} would be a quotient of $\Gamma^\CC(X)$ and therefore  $\Gamma^\CC(X)$ cannot be reduced.
\end{rem}

We next consider some examples of free proalgebraic groups. The first three examples explain how our notion of free proalgebraic groups encompasses previously known constructions.

\begin{ex}
    Let $\mathtt{C}$ be a formation of finite groups and let $\CC=\mathtt{C}_k$ be the corresponding formation of algebraic groups over $k$ (Section \ref{subsec: Formations}). Let $X$ be a set and let $\mathtt{\Gamma}^\mathtt{C}(X)$ be the free pro-$\mathtt{C}$-group on $X$ (\cite[Def. 17.4.1]{FriedJarden:FieldArithmetic}). Then $\Gamma^\CC(X)=(\mathtt{\Gamma}^\mathtt{C}(X))_k$. Since $G(k)=G(L)$ for any field extension $L$ of $k$ and $\CC$-group $G$, one actually has $\Gamma^\CC_{L/k}(X)=(\mathtt{\Gamma}^\mathtt{C}(X))_k$ for any field extension $L$ of $k$.
    \end{ex}

\begin{ex}
 Let $k$ be an algebraically closed field, $X$ a finite set and $\CC$ the formation of all algebraic groups. Then $\Gamma^\CC(X)$ is the proalgebraic completion (\cite[Def. 4]{BassLubotzkyMagidMozes:TheProalgebraicCompletionOfRigidGroups})  of the (abstract) free group on $X$. If $k$ has characteristic zero and $X$ has only one element, one can show that $\Gamma^\CC(X)=\Ga\times D(k^\times)$. See \cite{Mathoverflow:FreeAbelianRankOneInAffineGroupSchemes} or \cite[Example 1]{BassLubotzkyMagidMozes:TheProalgebraicCompletionOfRigidGroups}.
\end{ex}

\begin{ex} \label{ex: free prounipotent}
    Free prounipotent groups over an algebraically closed field of characteristic zero were introduced in \cite{LubotzkyMagid:CohomologyOfUnipotentAndPropunipotentGroups} and further studied in \cite{LubotzkyMagid:FreeProunipotentGroups}.
    The construction and definition of free prounipotent groups given in \cite{LubotzkyMagid:CohomologyOfUnipotentAndPropunipotentGroups} does not quite agree with ours, however in \cite[Prop. 2.2]{LubotzkyMagid:CohomologyOfUnipotentAndPropunipotentGroups} the following mapping property of $U(X)$, the free prounipotent group on a set $X$ (in the sense of \cite{LubotzkyMagid:CohomologyOfUnipotentAndPropunipotentGroups}) is established: for any unipotent algebraic group $G$, there is a bijection between the set of morphisms $U(X)\to G$ with the set of subsets $\{g_x|\ x\in X\}$ of $G(k)$ where $g_x=1$ for almost all $x\in X$, such that a morphism $\psi\colon U(X)\to G$ corresponds to $\{\psi(x)|\ x\in X \}$. It is clear from the construction given in \cite[Section 2]{LubotzkyMagid:CohomologyOfUnipotentAndPropunipotentGroups} that the map $X\to U(X)$ converges to $1$. Thus the mapping property from \cite[Prop. 2.2]{LubotzkyMagid:CohomologyOfUnipotentAndPropunipotentGroups} is equivalent to the mapping property of Theorem \ref{theo: free pro-alg group} (Remark~\ref{rem: can reduce to algebraic}). In other words, over an algebraically closed field of characteristic zero, $U(X)=\Gamma^\CC(X)$, where $\CC$ is the class of all unipotent algebraic groups.

    If $G$ is a unipotent algebraic group over an algebraically closed field $k$ of characteristic zero and if $x\in G(k)\smallsetminus 1$, then $\langle x\rangle\simeq \Ga$. (Note that $\langle x\rangle$ is abelian and unipotent and therefore isomorphic to $\Ga^n$ for some $n$, but then necessarily $n=1$.)  It follows that $\Gamma^\CC(X)=\Ga$ if $X$ is a one element set.

    If $k$ is not algebraically closed, $\Gamma^\CC(X)$ is more complicated. Indeed, if $k=\mathbb{Q}$ and $x=(\sqrt{2},\sqrt{3})\in \Ga^2(\kb)$, then $\langle x\rangle=\Ga^2$. Thus $\Ga^2$ is a quotient of $\Gamma^\CC(X)$.
\end{ex}

%\begin{ex} \label{ex: free abelian unipotent}
%   Let $k$ be a field of characteristic zero and $\CC$ the formation of all abelian unipotent algebraic groups. Thus an algebraic group is in $\CC$ if and only if it is isomorphic to $\Ga^n$ for some $n$ (\cite[Cor. 16.15]{Milne:AlgebraicGroupsTheTheoryOfGroupSchemesOfFiniteTypeOverAField}). For any set $X$ we have $\Gamma^\CC_{k/k}(X)=\prod_X\Ga$.
%\end{ex}
%\begin{proof}
%   We define $\iota\colon X\to\prod_X\Ga(k)$ by
%   $$\iota(x)_y=\begin{cases}
%   1 \text{ if } y=x \\
%   0 \text{ otherwise. }
%   \end{cases}$$
%   The kernels of the projections $\prod_X\Ga\to\prod_Y\Ga$, where $Y$ is a finite subset of $X$ are a neighborhood basis at $1$ for $\prod_X\Ga$. Every such kernel contains almost all elements of $\iota(X)$. Thus $\iota$ converges to one.
%
%
%   Let $G$ be a pro-$\CC$-group and $\varphi\colon X\to G(k)$ a map converging to $1$. Then $G$ is isomorphic to $\prod_Z\Ga$ for some set $Z$. We define $\psi\colon \prod_X\Ga\to\prod_Z\Ga$ by
%   $$\psi((r_x)_{x\in X})=\left(\sum_{x\in X} r_x\pi_z(\varphi(x))\right)_{z\in Z}\in \prod_Z \Ga(R)$$ for any $k$-algebra $R$ and $(r_x)_{x\in X}\in\prod_X \Ga(R)$, where $\pi_z\colon \prod_Z\Ga\to \Ga$ is the projection that picks out the $z$-coordinate. The above sum is finite because $\varphi$ converges to $1$. Then $\psi$ is the unique morphism such that
%$$
%\xymatrix{
%   X \ar^{\iota}[rr] \ar_{\varphi}[rd] &  & \prod_X\Ga(k)  \ar^{\psi_k}[ld] \\
%   & \prod_Z\Ga(k) &
%}
%$$
%commutes.
% \end{proof}

\begin{ex} \label{ex: free abelian unipotent}
    Let $k$ be a field of characteristic zero and $\CC$ the formation of all abelian unipotent algebraic groups. Thus an algebraic group is in $\CC$ if and only if it is isomorphic to $\Ga^n$ for some $n$ (\cite[Cor. 16.15]{Milne:AlgebraicGroupsTheTheoryOfGroupSchemesOfFiniteTypeOverAField}). For any set $X$ and $k$-algebra $R$ we have $\Gamma^\CC_{R/k}(X)=\prod_Y\Ga$, where $Y$ is a set of cardinality $|X|\cdot\dim_k(R)$.
\end{ex}
\begin{proof}
    Let $B$ be a $k$-basis of $R$ and set $Y=X\times B$. We define $\iota\colon X\to\prod_Y\Ga(R)$ by
    $$\iota(x)_y=\begin{cases}
    b \text{ if } y=(x,b), \\
    0 \text{ otherwise. }
    \end{cases}$$
    The kernels of the projections $\prod_Y\Ga\to\prod_{Y'}\Ga$, where $Y'$ is a finite subset of $Y$ are a neighborhood basis at $1$ for $\prod_Y\Ga$. Every such kernel contains almost all elements of $\iota(X)$. Thus $\iota$ converges to one.

    Let $G$ be a pro-$\CC$-group and $\varphi\colon X\to G(R)$ a map converging to $1$. Then $G$ is isomorphic to $\prod_Z\Ga$ for some set $Z$.
    For $x\in X$ and $z\in Z$ we can write $\varphi(x)_z=\sum_{b\in B}\lambda(b,x,z)b\in R$ for uniquely determined  $\lambda(b,x,z)\in k$. We define $\psi\colon \prod_Y\Ga\to\prod_Z\Ga$ by
    $$\psi((s_y)_{y\in Y})=\left(\sum_{x\in X}\sum_{b\in B} \lambda(b,x,z)s_{(x,b)}\right)_{z\in Z}\in \prod_Z \Ga(S)$$ for any $k$-algebra $S$ and $(s_y)_{y\in Y}\in\prod_Y \Ga(S)$. The above sum is finite because $\varphi$ converges to $1$. Then $\psi$ is the unique morphism such that
    $$
    \xymatrix{
        X \ar^{\iota}[rr] \ar_{\varphi}[rd] &  & \prod_Y\Ga(R)  \ar^{\psi_R}[ld] \\
        & \prod_Z\Ga(R) &
    }
    $$
    commutes.
\end{proof}

%In the following example we describe the free prodiagonalizable groups explicitly.
\begin{ex} \label{ex: free prodiagonalizable group}
    Let $\CC$ be the class of all diagonalizable algebraic groups, $R$ a $k$-algebra and $X$ a set. Then $\Gamma^\CC_{R/k}(X)=D(\oplus_X R^\times)$. Here $\oplus_X R^\times$ is the direct sum of $|X|$ copies of the abelian group $R^\times$.
\end{ex}
\begin{proof}
Let us define $\iota\colon X\to D(\oplus_X R^\times)(R)=\Hom(\oplus_X R^\times,R^\times)$ as the map that sends $x\in X$ to the projection $\pi_x\colon \oplus_X R^\times\to R^\times$ that picks out the $x$-component. We need to check that $\iota$ converges to $1$.

In general, if $M$ is an abelian group, the inclusion map $N=D(M')\hookrightarrow D(M)$ of a coalgebraic subgroup into $D(M)$, corresponds to a quotient map $M\twoheadrightarrow M'$ of abelian groups with a finitely generated kernel. So $\iota$ converges to $1$, because for every finitely generated subgroup $A$ of $\oplus_X R^\times$, $A\subseteq\ker(\pi_x)$ for almost all $x\in X$.

%A closed subgroup $D(M')$ of $D(\oplus_X R^\times)$ such that $\iota(X)\subseteq D(M')(R)$ corresponds to a surjection $f\colon\oplus_X R^\times\twoheadrightarrow M'$ of abelian groups such that $\ker(f)\subseteq\ker(\pi_x)$ for all $x\in X$. But then $\ker(f)$ must be trivial and therefore $\langle\iota(X)\rangle=D(\oplus_X R^\times)$.

Let $G=D(M)$ be a pro-$\CC$-group and $\varphi\colon X\to G(R)$ a map converging to $1$. We have to find $\psi\colon D(\oplus_X R^\times)\to G$ such that
    \begin{equation} \label{eqn: universal prop diagonalizable}
\xymatrix{
    X \ar^{\iota}[rr] \ar_{\varphi}[rd] &  & D(\oplus_X R^\times)(R) \ar^{\psi_R}[ld] \\
    & G(R) &
}
\end{equation}
commutes. This corresponds to finding a morphism $\widetilde{\psi}\colon M\to \oplus_X R^\times$  of abelian groups such that

    \begin{equation} \label{eqn: ex abelian groups}
\xymatrix{
    M \ar^{\widetilde{\psi}}[rr] \ar_{\varphi(x)}[rd] &  & \oplus_X R^\times \ar^{\pi_x}[ld] \\
    & R^\times &
}
\end{equation}
commutes for all $x\in X$. Clearly, the unique solution to this problem is given by $\widetilde{\psi}(m)=\oplus_{x\in X}\varphi(x)(m)$ for $m\in M$. Note that this sum is finite because  $\varphi$ converges to $1$.
\end{proof}

We would like to warn the reader that the formation of free pro-$\CC$-groups is not compatible with extension of the base field in the sense that, for example, the free pro-diagonalizable group over $k'$ on $X$ can be obtained from the free pro-diagonalizable group over $k$ by base extension from $k$ to $k'$.

We conclude this section with two lemmas that will be needed later. The following lemma is fundamental for showing that certain embedding problems for free pro-$\CC$-groups are solvable.

\begin{lemma} \label{lemma: finite generation}
    Let $k$ be field of characteristic zero and let $G$ be an algebraic group. Then there exists a finite subset $X$ of $G(\kb)$ such that $G=\langle X\rangle$.
\end{lemma}
\begin{proof}
    A proof, under the additional assumption that $k$ is algebraically closed, can be found in \cite[Lemma 5.13]{SingerPut:differential}.  But if $X$ generates $G_{\kb}$, then $X$ also generates $G$.
\end{proof}

The above lemma fails in positive characteristic, even for smooth groups, for example $\Ga$ is not finitely generated in positive characteristic.

\begin{lemma} \label{lemma: free group not algebraic}
    Let $k$ be a field of characteristic zero, $X$ an infinite set and $\CC$ a formation. Then $\Gamma^\CC(X)$ is not an algebraic group.
\end{lemma}
\begin{proof}
    By assumption there exists a non-trivial $\CC$-group $G$. Since $k$ has characteristic zero $G(\kb)\neq 1$ and we can choose $g\in G(\kb)$ with $g\neq 1$. Moreover, by Lemma \ref{lemma: finite generation}, we can find $g_1,\ldots,g_n\in G(\kb)$ such that $G=\langle g_1,\ldots,g_n\rangle$. Pick $x_1,\ldots,x_n\in X$ and for every finite subset $Y$ of $X$ containing $x_1,\ldots,x_n$ define a map $\varphi_Y\colon X\twoheadrightarrow G(\kb)$ by $\varphi_Y(x_i)=g_i$, $(i=1,\ldots,n)$, $\varphi_Y(y)=g$ for $y\in Y\smallsetminus\{x_1,\ldots,x_n\}$ and $\varphi_Y(x)=1$ for $x\in X\smallsetminus Y$. Then $\varphi_Y$ converges to $1$ and therefore defines an epimorphism $\f_Y\colon\Gamma^\CC(X)\twoheadrightarrow G$.

    Let $\{x_1,\ldots,x_n\}\subsetneqq Y_1\subsetneqq Y_2\subsetneqq\ldots$ be a strictly ascending chain of finite subsets of $X$. For $i\geq 1$ set $N_i=\bigcap_{j=1,\ldots,i}\ker(\f_{Y_i})$. If $y\in Y_{i+1}\smallsetminus Y_i$, then $\iota(y)\notin N_{i+1}(\kb)$ but $\iota(y)\in N_i(\kb)$. Thus $N_1\supsetneqq N_2\supsetneqq\ldots$ is a strictly descending chain of closed subgroups of $\Gamma^\CC(X)$. Hence $\Gamma^\CC(X)$ cannot be an algebraic group.
\end{proof}

\section{Saturated pro-$\CC$-groups}

The main goal of this section is to characterize free pro-$\CC$-groups in terms of embedding problems. To achieve this goal we first study certain pro-$\CC$-groups that solve ``many'' embedding problems, we call these groups \emph{saturated}.

\subsection{The rank of a proalgebraic group}

As we will show, for any formation $\CC$ there exists up to isomorphism a unique saturated pro-$\CC$-group $\Gamma$, once we fix the ``size'' of $\Gamma$ and this size is sufficiently big. The goal of this subsection is to make precise this notion of size by introducing the rank of a proalgebraic group.

%
%In this section we introduce the rank of a proalgebraic group; a rather coarse measure for the size of a proalgebraic group.

The rank of a profinite group $G$ can be defined as the smallest cardinal number $\kappa$ such that $G$ has a set of $\kappa$ (topological) generators that converges to $1$.
At first sight, it may seem plausible to define the rank $r(G)$ of a proalgebraic group $G$, say over an algebraically closed field $k$, as the smallest cardinal $\kappa$ such that there exists a subset $X$ of $G(k)$ converging to $1$ with $\langle X\rangle=G$. However, this approach has drawbacks and yields pathologies. For example, over a field of positive characteristic, there may not exist such an $X$. (This is for example the case for $\Ga$.)

Also, for an infinite index set $I$, $r(\prod_{i\in I}\Ga)$ and $r(\prod_{i\in I}\Gm)$ behave quite differently. If $k$ has characteristic zero, then $r(\prod_{i\in I}\Ga)$ is infinite. On the other hand, if $k$ has transcendence degree greater or equal to $|I|$, then $r(\prod_{i\in I}\Gm)=1$ since $(g_i)_{i\in I}$ is a generator if $(g_i)_{i\in I}$ are algebraically independent over the prime field.
%The latter example also shows that $r$ is not stable under extension of the base field $k$.

Below we will define the rank $\rank(G)$ for any proalgebraic group $G$. It satisfies the following equalities $\rank(\prod_{i\in I}\Ga)=\rank(\prod_{i\in I}\Gm)=|I|$.

\begin{prop} \label{prop: rank}
    Let $G$ be a proalgebraic group that is not algebraic. Then the following cardinals are equal.
    \begin{enumerate}
        \item The dimension of $k[G]$ as a $k$-vector space.
        \item The smallest cardinal $\kappa$ such that $k[G]$ can be generated as a $k$-algebra by $\kappa$ elements.
        \item The smallest cardinal $\kappa$ such that $k[G]$ is the directed union of $\kappa$ finitely generated Hopf subalgebras.
        \item The smallest cardinal $\kappa$ such that there exists a neighborhood basis at $1$ for $G$ of cardinality $\kappa$.
        %       \item The cardinality of the set of all coalgebraic subgroups of $G$.
        \item The smallest cardinal $\kappa$ such that there exists a directed set $I$ of cardinality $\kappa$ and algebraic groups $G_i$ with $G\simeq\varprojlim_{i\in I}G_i$.
    \end{enumerate}
\end{prop}
\begin{proof}
    For $i=1,\ldots,5$ let $\kappa_i$ denote the cardinal defined in point $(i)$ above. Clearly, $\kappa_2\leq\kappa_1$. On the other hand, if $(f_i)_{i\in I}$ generates $k[G]$ as a $k$-algebra, then $\bigcup_{J\subseteq I}F_J$ generates $k[G]$ as a $k$-vector space, where, for every finite subset $J$ of $I$ the set $F_J$ consists of all monomials in $(f_j)_{j\in J}$. As $\bigcup_{J\subseteq I}F_J$ is the union of $|I|$ countable subsets, we see that $k[G]$ can be generated by $|I|$ elements as a $k$-vector space. Thus $\kappa_1\leq\kappa_2$.

    Assume that $(f_i)_{i\in I}$ generates $k[G]$ as a $k$-algebra. For every finite subset $J$ of $I$ let $A_J$ be the smallest Hopf subalgebra of $k[G]$ that contains all the $f_j$ for $j\in J$. Then $A_J$ is finitely generated as a $k$-algebra (\cite[Section 3.3]{Waterhouse:IntroductiontoAffineGroupSchemes}) and the $A_J$'s form a directed system whose union is $k[G]$. Thus $\kappa_3\leq\kappa_2$. On the other hand, choosing a finite set of generators for every Hopf subalgebra in a directed system shows that $\kappa_2\leq\kappa_3$.
    As finitely generated Hopf subalgebras correspond to coalgebraic subgroups we have $\kappa_3=\kappa_4$.

    If $\N$ is a neighborhood basis at $1$ for $G$, then $G\simeq\varprojlim_{N\in\N} G/N$. Conversely, if $G\simeq\varprojlim_{i\in I}G_i$, and $N_i$ is the kernel of the projection $G\to G_i$, then the $N_i$ are a neighborhood basis at $1$ for $G$. Thus $\kappa_4=\kappa_5$.
\end{proof}

We would like to define the \emph{rank} of a proalgebraic group as the quantity characterized in Proposition \ref{prop: rank}. We note that this notion of rank has nothing to do with the classical notion of rank for algebraic groups defined as the dimension of a Cartan subgroup. For an algebraic group $G$, the quantities defined in (iii), (iv) and (v) of Proposition~\ref{prop: rank} are all equal to $1$, but the quantities defined in (i) and (ii) may be distinct and different from $1$. For example, for $G=\Ga^2$, the quantity defined in (i) is $|\nn|$, whereas the quantity defined in (ii) is $2$. The exact definition of the rank of an algebraic group is largely irrelevant for what follows (as long as it is finite). The notationally most convenient choice is to set it equal to $1$ for a non-trivial algebraic group and equal to $0$ for the trivial group. We therefore make the following definition.

\begin{defi}
 The \emph{rank} $\rank(G)$ of a non-trivial proalgebraic group $G$ is the smallest cardinal $\kappa$ such that there exists a neighborhood basis at $1$ for $G$ of cardinality $\kappa$. The rank of the trivial algebraic group is $0$.
\end{defi}

%To keep the notation in our statements simple, we define the rank of the trivial group to be zero.

Thus the rank of a proalgebraic group $G$ is finite if and only if $G$ is algebraic. % Proposition \ref{prop: rank} gives several equivalent description of the rank of a proalgebraic group that is not algebraic.
Let $\mathtt{G}$ be an infinite profinite group and let $G=\mathtt{G}_k$ denote the corresponding proalgebaic group.
The cardinality of a neighborhood basis at $1$ for $\mathtt{G}$ consisting of open normal subgroups does not depend on the neighborhood basis and agrees with the minimal cardinality of any neighborhood basis at $1$ for $\mathtt{G}$ (not necessarily consisting of groups). This cardinal number is denoted by $w_0(\mathtt{G})$ in \cite{RibesZalesskii:ProfiniteGroups}. Since the coalgebraic subgroups of $G$ are in one-to-one correspondence with the open normal subgroups of $\mathtt{G}$, it follows that $\rank(G)=w_0(\mathtt{G})$. In particular, if $G$ is not (topologically) finitely generated, it follows that $\rank(G)=\rank(\mathtt{G})$. (See \cite[Cor. 2.6.3]{RibesZalesskii:ProfiniteGroups} or \cite[Prop. 17.1.2]{FriedJarden:FieldArithmetic}.) So the rank for proalgebraic groups generalizes the notion of rank for profinite groups.\

%However, certain statements familiar for the rank of profinite groups do not hold for the rank of proalgebraic groups. For example, it is in general not true that the rank of a non-finitely generated proalgebraic group equals the cardinality of the set of all coalgebraic subgroups.
%%
%
%
%
% that is not (topologically) finitely generated and $G=\mathtt{G}_k$ denotes the corresponding proalgebaic group, then $\rank(\G)$, the rank of the profinite group $\G$ agrees with $\rank(G)$, the rank of the proalgebraic group $G$, because the coalgebraic subgroups of $G$ are in bijection with the open normal subgroups of $\G$ and the cardinality of the latter set is known to equal $\rank(\G)$ ().

\begin{ex} \label{ex: rank for fibre product}
        Let $(G_i)_{i\in I}$ be an infinite family of non-trivial algebraic groups and $G=\prod_{i\in I} G_i$. Then $\rank(G)=|I|$.
    More generally, for an arbitrary set $I$ let $\f_i\colon G_i\twoheadrightarrow H$ ($i\in I$) be a family of epimorphisms of proalgebraic groups. Let $\prod_{i\in I} (G_i\twoheadrightarrow H)$ denote the functor from the category of $k$-algebras to the category of groups defined by $$\left(\prod_{i\in I} (G_i\twoheadrightarrow H)\right)(R)=\{(g_i)_{i\in I}|\ g_i\in G_i(R),\ \f_i(g_i)=\f_j(g_j) \ \forall \ i,j \in I\}$$
    for any $k$-algebra $R$. This functor is representable, i.e., a proalgebraic group. Indeed, it is represented by $\varinjlim         k[G_{i_1}]\otimes_{k[H]}\ldots\otimes_{k[H]}k[G_{i_n}]$, where the direct limit is taken over all finite subsets $\{i_1,\ldots,i_n\}$ of $I$ ordered by inclusion. We call $\prod_{i\in I} (G_i\twoheadrightarrow H)$ the \emph{fibre product} of the $\f_i$. If $H=1$ this simply reduces to the product. If $I=\{1,2\}$ one usually writes $G_1\times_H G_2$ for the corresponding fibre product.

    It follows from the above description of the coordinate ring of the fibred product that if $\kappa$ is an infinite cardinal such that $\rank(G_i)\leq \kappa$ and $|I|\leq\kappa$, then  $\rank(\prod_{i\in I} (G_i\twoheadrightarrow H))\leq \kappa.$
\end{ex}

\begin{ex}
    If $M$ is a non-trivial abelian group, then $\rank(D(M))=|M|$ unless $M$ is finitely generated, in which case $\rank(D(M))=1$.
\end{ex}

\begin{lemma} \label{lemma: rank for subgroups and quotients}
    Let $G$ be a proalgebraic group. If $H$ is a closed subgroup or a quotient of $G$, then its rank is such that $\rank(H)\leq \rank(G)$.
\end{lemma}
\begin{proof}
    This is clear if $G$ is algebraic and if $G$ is not algebraic the statement follows from Proposition~\ref{prop: rank}, using e.g., characterization (i).
\end{proof}

\begin{lemma} \label{lemma: ir for algebraic kernel}
    Let $\f\colon G\twoheadrightarrow H$ be an epimorphism of proalgebraic groups with algebraic kernel. Then we have $\ir(G)=\ir(H)$.
\end{lemma}
\begin{proof}
    By Lemma \ref{lemma: rank for subgroups and quotients} it suffices to show that $\rank(G)\leq\rank(H)$. By Corollary \ref{cor: intersection 1} there exists a coalgebraic subgroup $M$ of $G$ such that $M\cap\ker(\f)=1$.
    Let $\N$ be a neighborhood basis at $1$ for $H$. We claim that the set $\N'=\{\f^{-1}(N)\cap M|\ N\in\N\}$ is a neighborhood basis at $1$ for $G$. First of all, note that $\f^{-1}(N)\cap M\leq G$ is coalgebraic because $\f^{-1}(N)$ and $M$ are coalgebraic. Furthermore, $\N'$ is downward directed because $\N$ is and
    $$\bigcap_{N\in\N}(\f^{-1}(N)\cap M)=\f^{-1}(\bigcap_{N\in\N} N)\cap M=\ker(\f)\cap M=1.$$
    Thus $\ir(G)\leq\ir(H)$.
\end{proof}

\begin{lemma} \label{lemma: ir for intersection}
    Let $G$ be a proalgebraic group and $(N_i)_{i\in I}$ a family of closed normal subgroups of $G$. If $\ir(G/N_i)\leq\kappa$ for all $i\in I$ and $|I|\leq \kappa$, then $\ir(G/\cap N_i)\leq\kappa$.
\end{lemma}
\begin{proof}
The statement is clear for finite $\kappa$. So let us assume that $\kappa$ is infinite. We have the inclusions $k[G/N_i]\subseteq k[G/\cap N_i]\subseteq k[G]$ and all the $k[G/N_i]$ generate $k[G/\cap N_i]$ as a $k$-algebra. Thus joining the generators of the individual $k[G/N_i]$, gives generators for $k[G/\cap N_i]$. Consequently, $k[G/\cap N_i]$ can be generated by $\kappa\cdot\kappa=\kappa$ elements. Thus the claim follows from Proposition \ref{prop: rank}.
\end{proof}

%For a proalgebraic group $G$, let $\ir(G)$ denote the smallest cardinal $\kappa$ such that $k[G]$ is the directed union of $\kappa$ finitely generated Hopf-subalgebras. An alternative description of $\ir(G)$ is the following: A set $\N$ of coalgebraic subgroups of $G$ is a \emph{base} for $G$ if
%for every coalgebraic subgroup $N$ of $G$, there exists $N'\in \N$ with $N'\leq N$. Then $\ir(G)$ is the smallest cardinal $\kappa$ such that there exists a base for $G$ of size $\kappa$.
%
%If $G$ is an algebraic group, then $\ir(G)=1$. If $G$ is not an algebraic group, then $\ir(G)$ is infinite and it is the smallest cardinal $\kappa$ such that $k[G]$ can be generated by $\kappa$ elements. ( more details)

The following proposition shows that the rank can also be understood as the minimal length of a subnormal series satisfying certain conditions.

\begin{prop} \label{prop: bound rank}
    Let $G$ be a proalgebraic group and $\mu$ an ordinal number. If there exists a chain $$G=N_0\geq N_1\geq\ldots\geq N_\lambda\geq\ldots\geq N_\mu=1$$
    of normal closed subgroups $N_\lambda$ of $G$ ($\lambda\leq\mu$) with
    \begin{enumerate}
        \item $N_\lambda/N_{\lambda+1}$ algebraic for all $\lambda<\mu$ and
        \item $N_\lambda=\bigcap_{\lambda'<\lambda}N_{\lambda'}$ if $\lambda$ is a limit ordinal,
    \end{enumerate}
    then $\rank(G)\leq|\mu|$. Moreover, there exists an ordinal number $\mu$ with $|\mu|=\rank(G)$ and a chain as above satisfying (i) and (ii). In particular, $\rank(G)=\min|\mu|$, where $\mu$ ranges over all ordinals for which there exists such a chain.
\end{prop}
\begin{proof}
    %   Let $\mu$ be the smallest ordinal number such that $|\mu|=\rank(G)$. Furthermore, let $\mathcal{U}$ be a neighborhood basis at $1$ for $G$ with $|\mathcal{U}|=\rank(G)=|\mu|$. So we may write $\mathcal{U}=\{U_\lambda|\ \lambda<\mu\}$ and set $N_\lambda=\bigcap_{\lambda'<\lambda}U_{\lambda'}$ for $\lambda\leq\mu$.  Clearly the $N_\lambda$'s form a descending chain of normal closed subgroups of $G$ such that $N_0=G$ and $N_\mu=1$. Moreover, $N_\lambda/N_{\lambda+1}=N_\lambda/(N_\lambda\cap U_{\lambda+1})=U_{\lambda+1}N_\lambda/U_{\lambda+1}$ is algebraic, because $U_{\lambda+1}N_\lambda/U_{\lambda+1}$ embeds into $G/U_{\lambda+1}$.
    We will show by transfinite induction that $\rank(G/N_\lambda)\leq |\lambda|$ for every $\lambda\leq\mu$. The initial case $\lambda=0$ is true by definition. If $\lambda=\lambda'+1$ is a successor ordinal, then $G/N_{\lambda'+1}\twoheadrightarrow G/N_{\lambda'}$ has kernel $N_{\lambda'}/N_{\lambda'+1}$. Thus $\rank(G/N_\lambda)=\rank(G/N_{\lambda'})\leq |\lambda'|\leq |\lambda|$ by Lemma \ref{lemma: ir for algebraic kernel}.
    If $\lambda$ is a limit ordinal, then $\rank(G/N_\lambda)\leq |\lambda|$ by Lemma \ref{lemma: ir for intersection}.

    Let us now show that there exists a chain with $|\mu|=\rank(G)$. Let $\N$ be a neighborhood basis at $1$ for $G$ with $|\N|=\rank(G)$ and let $\mu$ be an ordinal number with $|\mu|=|\N|$. So we may write $\N=\{U_\lambda|\ \lambda<\mu\}$. For $\lambda\leq \mu$ we set  $N_\lambda=\bigcap_{\lambda'<\lambda} U_\lambda$. Clearly the $N_\lambda$'s form a descending chain of normal closed subgroups of $G$ such that $N_0=N$ and $N_\mu=1$. Moreover $N_\lambda/N_{\lambda+1}=N_\lambda/(N_\lambda\cap U_{\lambda+1})=U_{\lambda+1}N_\lambda/U_{\lambda+1}$ is algebraic, because $U_{\lambda+1}N_\lambda/U_{\lambda+1}$ embeds into $N/U_{\lambda+1}$.
    Clearly (ii) is satisfied.
\end{proof}

Our next goal is to determine the rank of a free pro-$\CC$-group. The following lemma yields a lower bound.

\begin{lemma} \label{lemma: ir is upper bound for X}
    Let $G$ be a proalgebraic group that is not algebraic and let $R$ be a $k$-algebra. If $X\subseteq G(R)$ converges to $1$, then $|X|\leq \ir(G)$.
\end{lemma}
\begin{proof}
    %We can assume without loss of generality that $X$ is infinite.
     Let $\N$ be a neighborhood basis at $1$ for $G$. We have $X=\bigcup_{N\in\N} (X\smallsetminus N(R))$ (assuming w.l.o.g. $1\notin X$) with $X\smallsetminus N(R)$ finite. So $|X|\leq|\N|$.
\end{proof}

The next lemma yields an upper bound for the rank of free pro-$\CC$-groups.

\begin{lemma}
    Let $\CC$ be a formation and let $X$ be a set. Then $\Gamma^\CC(X)$ has at most $|\kb||X|$ coalgebraic subgroups.
%   Moreover, if there exists an algebraic group in $\CC$ of positive dimension, then $\Gamma^\CC(X)$ has exactly $|\kb||X|$ coalgebraic subgroups.
\end{lemma}
\begin{proof}
    Let $\N$ denote the set of all coalgebraic subgroups of $\Gamma^\CC(X)$. These all arise as kernels of some epimorphism $\f\colon\Gamma^\CC(X)\twoheadrightarrow H$ for some $\CC$-group $H$. Such a $\f$ is determined by choosing a finite subset $\{x_1,\ldots,x_n\}$ of $X$ and the images $\f_{\kb}(x_1),\ldots,\f_{\kb}(x_n)\in H(\kb)$. For a fixed $H$, there are $|X||H(\kb)|$ choices. There are, up to isomorphism, at most $|\kb|$ algebraic groups over $k$. Thus, also allowing $H$ to vary, we see that there are at most $|\kb||X|$ coalgebraic subgroups.
%   Let us assume that there exists a positive dimensional $\CC$-group $H$. Then $H\times H$ is also a $\CC$-group. Let us fix an $x\in X$ and $1\neq h\in H(\kb)$. For every $x'\in X\smallsetminus\{x\}$ let $\psi_{x'}\colon \Gamma^\CC(X)\twoheadrightarrow H\times H$ be the map determined by $x\mapsto (1,0)$, $x'\mapsto (0,1)$ and $x''\mapsto (0,0)$ for $x''\in X\smallsetminus\{x,x'\}$.
%   There are $|k|$ closed subgroups of $\Ga^2$ that do not contain $(0,1)$. These correspond to $|k|$ coalgebraic subgroups of $G$ that do not contain $x'$ but all $x''\in X\smallsetminus\{x,x'\}$. So we have at least $|k||X|$ coalgebraic subgroups.
%
%
%
%
%
%%
%%      There are $|k|$ closed subgroups of $\Ga^2$ that do not contain $(0,1)$. These correspond to $|k|$ coalgebraic subgroups of $G$ that do not contain $x'$ but all $x''\in X\smallsetminus\{x,x'\}$. So we have at least $|k||X|$ coalgebraic subgroups.
%%
%%   of positive dimension, then there are most $|X||\overline{k}|$ coalgebraic subgroups, since $|\CC|\leq |\overline{k}|$.
%%
%%  It remains to show that these upper bounds are actually attained:
\end{proof}

Combining the previous two lemmas we obtain:

\begin{cor} \label{cor: inequality for rank of free group}
    Let $\CC$ be a formation and let $X$ be a set such that $\Gamma^\CC(X)$ is not algebraic. Then
    \begin{equation} \label{eqn: inequ for rank}
    |X|\leq\rank(\Gamma^\CC(X))\leq |\kb||X|.
    \end{equation}
 In particular, if $|X|\geq|\kb|$, then $\rank(\Gamma^\CC(X))=|X|$. \qed
\end{cor}
Both extreme cases of inequality (\ref{eqn: inequ for rank}) may occur: If $X$ is an infinite set and $\CC$ is the formation of all abelian unipotent groups over an algebraically closed field of characteristic zero, then $\rank(\Gamma^\CC(X))=|X|$ by Example \ref{ex: free abelian unipotent}. On the other hand, if $\CC$ is the formation of all diagonalizable algebraic groups, then $\rank(\Gamma^\CC(X))=|\kb||X|$ by Example \ref{ex: free prodiagonalizable group}.

In some rare cases $\Gamma^\CC(X)$ can indeed be an algebraic group and then the inequality $|X|\leq\rank(\Gamma^\CC(X))$ from Corollary \ref{cor: inequality for rank of free group} may not hold. For example, if $k$ is an algebraically closed field of characteristic zero, $X$ a finite set and $\CC$ the formation of all abelian unipotent groups, then $\Gamma^\CC(X)=\Ga^n$, where $n=|X|$, by Example \ref{ex: free abelian unipotent}. Also, if $k$ is a field of positive characteristic and $\CC$ is the formation of all infinitesimal algebraic groups, then $\Gamma^\CC(X)=1$ for any set $X$ (since $G(\kb)=1$ for any infinitesimal algebraic group $G$). % To show that in characteristic zero, such a pathology cannot happen, we will need the following lemma.
However, if $k$ has characteristic zero and $X$ is infinite, then $\Gamma^\CC(X)$ is not algebraic by Lemma \ref{lemma: free group not algebraic}. We thus obtain from Corollary \ref{cor: inequality for rank of free group}:

\begin{cor} \label{cor: rank of free is X}
    Let $k$ be a field of characteristic zero, $\CC$ a formation and $X$ a set with $|X|\geq |k|$. Then $\rank(\Gamma^\CC(X))=|X|$. \qed
\end{cor}

We record one more lemma for later use.

\begin{lemma} \label{lemma: ker has infinitely many elements}
    Let $\CC$ be a formation and let $X$ be an infinite set. If $\f\colon\Gamma^\CC(X)\twoheadrightarrow H$ is an epimorphism such that $\ir(H)<|X|$, then $|X|$ many elements of $X$ map into $\ker(\f)(\overline{k})$.
\end{lemma}
\begin{proof}
    If $\ir(H)$ is finite the statement is obvious, so let us assume that $\ir(H)$ is infinite.

    Let $\N$ be a neighborhood basis at $1$ for $H$ with $|\N|=\rank(H)$. Since $\ker(\f)=\bigcap_{N\in\N}\f^{-1}(N)$, we have
    $$X\smallsetminus\iota^{-1}(\ker(\f)(\overline{k}))=\bigcup_{N\in\N}(X\smallsetminus\iota^{-1}(\f^{-1}(N)(\overline{k}))),$$
    where $\iota\colon X\to\Gamma^\CC(X)(\overline{k})$ is the map from Theorem \ref{theo: free pro-alg group}.
    The right hand side has cardinality bounded by $|\N|=\ir(H)<|X|$. Thus $|X\cap\iota^{-1}(\ker(\f)(\overline{k}))|=|X|$.
\end{proof}

\subsection{The dimension of a proalgebraic group}

For our characterization of saturated pro-$\CC$-groups we will also need another measure for the size of a proalgebraic group that is different from the rank. Let $G$ be a proalgebraic group. Then the nilradical $\ida$ of $k[G^o]$ is a prime ideal and we can define the dimension $\dim(G)$ of $G$ as the transcendence degree over $k$ of the field of fractions of $k[G^o]_{\operatorname{red}}=k[G^o]/\ida=k[(G^o)_{\operatorname{red}}]$. If $G$ is an algebraic group this is simply the usual dimension of $G$.

\begin{lemma} \label{lemma: dim for subgroups and quotients}
    Let $G$ be a proalgebraic group. If $H$ is a closed subgroup or a quotient of $G$, then we have $\dim(H)\leq\dim(G)$.
    \end{lemma}
\begin{proof}
    Let us first consider an epimorphism $G\twoheadrightarrow H$. Then we have an induced epimorphism $G^o\twoheadrightarrow H^o$.
 We therefore get inclusions $k[H^o]\subseteq k[G^o]$ and $k[H^o]_{\operatorname{\red}}\subseteq k[G^o]_{\operatorname{\red}}$. Thus $\dim(G)\geq\dim(H)$.

 If $H$ is a closed subgroup of $G$, then $H^o$ is a closed subgroup of $G^o$. We therefore have a surjection $k[G^o]\twoheadrightarrow k[H^o]$ that induces a surjection $k[G^o]_{\operatorname{red}}\twoheadrightarrow k[H^o]_{\operatorname{red}}$. Thus $\dim(H)\leq\dim(G)$.
\end{proof}

\begin{lemma} \label{lemma: dim for kernel zero dimensional}
    Let $\f\colon G\twoheadrightarrow H$ be an epimorphism of proalgebraic groups with $\dim(\ker(\f))=0$. Then we have $\dim(G)=\dim(H)$.
\end{lemma}
\begin{proof}
    By Lemma \ref{lemma: dim for subgroups and quotients} the kernel of the induced epimorphism $G^o\twoheadrightarrow H^o$ also has dimension zero. We can thus assume that $G$ and $H$ are connected. It suffices to show that every element of $k[G]_{\operatorname{red}}$ is algebraic over the field of fractions of $k[H]_{\operatorname{red}}$.
    Let $f\in k[G]$. Then $f\in k[G/N]$ for some coalgebraic subgroup $N$ of $G$. The kernel $\ker(\f)N/N=\ker(\f)/(\ker(\f)\cap N)$ of the epimorphism $G/N\twoheadrightarrow H/\f(N)$ of algebraic groups has dimension zero by Lemma \ref{lemma: dim for subgroups and quotients}. So $G/N$ and $H/\f(N)$ have the same dimension. Therefore every element in $k[G/N]_{\operatorname{red}}$ is algebraic over the field of fractions of $k[H/\f(N)]_{\operatorname{red}}$. In particular, the image of $f$ in $k[G/N]_{\operatorname{red}}\subseteq k[G]_{\operatorname{red}}$ is algebraic over the field of fractions of $k[H]_{\operatorname{red}}$.
%
%    $k[G/N]$ and $k[H/\f(N)]$-module. In particular, every element of $k[G/N]$ is integral over $k[H/\f(N)]$. This shows that $f$ is integral over $k[H]$ and it follows that the image of $f$ in $k[G]_{\operatorname{red}}$ is algebraic over the field of fractions of $k[H]_{\operatorname{red}}$.
\end{proof}

\begin{lemma} \label{lemma: dim bounded by rank}
    Let $G$ be a proalgebraic group that is not algebraic. Then $\dim(G)\leq\rank(G)$.
\end{lemma}
\begin{proof}
As $\dim(G)=\dim(G^o)$ and $\rank(G^o)\leq\rank(G)$ we may assume that $G$ is connected. We can choose a transcendence basis for the fraction field of $k[G]_{\operatorname{red}}$ that lies inside $k[G]_{\operatorname{red}}$ and lift it to a subset of $k[G]$. This subset is $k$-linearly independent and therefore the $k$-dimension of $k[G]$ is at least the dimension of $G$.
\end{proof}

\begin{lemma} \label{lemma: dim for intersection}
    Let $G$ be a proalgebraic group and $(N_i)_{i\in I}$ a family of normal closed subgroup such that \mbox{$G/\cap N_i$} has positive dimension. Then there exist an $i\in I$ such that $G/N_i$ has positive dimension.
\end{lemma}
\begin{proof}
    Since $(G/\cap N_i)^o=G^o/\cap(G^o\cap N_i)$ we may assume that $G$ is connected. As $k[G/\cap N_i]$ is generated by the subalgebras $k[G/N_i]$, the same property holds after modding out by the nilradical. Thus $\dim(G/N_i)>0$ for some $i\in I$.
\end{proof}

\begin{lemma} \label{lemma: dim of fibre product}
    Let $\alpha\colon G\twoheadrightarrow H$ be an epimorphism of algebraic groups such that $\dim(\ker(\alpha))>0$ and let $I$ be a set. Then $\dim(\prod_I(G\twoheadrightarrow H))\geq |I|$.
\end{lemma}
\begin{proof}
    Since $\prod_I(G\twoheadrightarrow H)$ contains the closed subgroup $\prod_{I}\ker(\alpha)$ it suffices to see that $\dim (\prod_{I}\ker(\alpha))\geq|I|$. But this follows rather directly from the definition.
\end{proof}

\subsection{Embedding problems}

We are now prepared to make precise the idea of a proalgebraic group that solves many embedding problems.

\begin{defi}
    An \emph{embedding problem} for a proalgebraic group $\Gamma$ consists of two epimorphisms $\alpha\colon G\twoheadrightarrow H$ and $\beta\colon \Gamma\twoheadrightarrow  H$ of proalgebraic groups. A \emph{solution} is an epimorphism $\f\colon\Gamma\twoheadrightarrow G$ such that $\beta=\alpha\f$, i.e.,
    \begin{equation}
    \label{eqn: embedding problem defi}
    \xymatrix{
        \Gamma \ar@{->>}^\beta[rd] \ar@{..>>}[d]\\
        G \ar@{->>}^-{\alpha}[r] & H
    }
    \end{equation}
    commutes. A morphism $\f\colon\Gamma\to G$ (that is not necessarily an epimorphism) such that $\beta=\alpha \f$ is called a \emph{weak solution}.
     We also refer to diagram (\ref{eqn: embedding problem defi}) as the embedding problem. If there exists a solution the embedding problem is called \emph{solvable}.
    The kernel of (\ref{eqn: embedding problem defi}) is $\ker(\alpha)$. If the kernel of (\ref{eqn: embedding problem defi}) is trivial, (\ref{eqn: embedding problem defi}) is called a \emph{trivial embedding problem}. (In this case, $\alpha^{-1}\beta$ is the unique solution.)
     The embedding problem  (\ref{eqn: embedding problem defi}) is \emph{algebraic} if $G$ (and therefore also $H$) are algebraic groups.
    %   An embedding problem is \emph{smooth} if $G$ (and therefore also $H$) is smooth.
    %

Let $\CC$ be a formation and assume that $\Gamma$ is a pro-$\CC$-group. Then  (\ref{eqn: embedding problem defi}) is a \emph{pro-$\CC$-embedding problem} if $G$ is a pro-$\CC$-group.  If $G$ is in $\CC$, then (\ref{eqn: embedding problem defi}) is a \emph{$\CC$-embedding problem}.
\end{defi}

%Note that if an algebraic embedding problem (\ref{eqn: embedding problem defi}) for $\Gamma$ is solvable, then $\alpha$ occurs in the canonical projective system defining $\Gamma$ (cf. ). Conversely, an occurrence of $\alpha$ in the canonical projective system corresponds to a solvable embedding problem.

It is instructive to reformulate (\ref{eqn: embedding problem defi}) in  terms of Hopf algebras. Given the two inclusions $k[H]\subseteq k[G]$ and $k[H]\subseteq k[\Gamma]$, we would like to find an embedding of $k[G]$ into $k[\Gamma]$ over $k[H]$.

Roughly speaking, we are interested in pro-$\CC$-groups that solve as many pro-$\CC$-embedding problems as possible. There does not exist a pro-$\CC$-group $\Gamma$ such that all pro-$\CC$-embedding problems (\ref{eqn: embedding problem defi}) for $\Gamma$ are solvable. For example, if $H=\Gamma$, $\beta$ is the identity map and (\ref{eqn: embedding problem defi}) is non-trivial, then (\ref{eqn: embedding problem defi}) does not have a solution. Also, if $\rank(G)>\rank(\Gamma)$, there cannot exist an epimorphism $\Gamma\to G$ (Lemma~\ref{lemma: rank for subgroups and quotients}). It turns out that imposing the restrictions $\rank(H)<\rank(\Gamma)$ and $\rank(G)\leq\rank(\Gamma)$ is sufficient to allow for the existence of a pro-$\CC$-group that solves all such pro-$\CC$-embedding problems. Moreover, for every cardinal number $\kappa\geq |\kb|$, there exists a unique such pro-$\CC$-group $\Gamma$ with $\rank(\Gamma)=\kappa$ (Theorems \ref{theo: existence of saturated groups} and \ref{theo: uniqueness}).

The following theorem summarizes several well-known characterizations of profinite groups that solve many embedding problems. See \cite[Section 3.5]{RibesZalesskii:ProfiniteGroups} or \cite[Section 25.1]{FriedJarden:FieldArithmetic}.

\begin{theo} \label{theo: saturation for profinite groups}
    Let $\mathtt{C}$ be a formation of finite groups and $\mathtt{\Gamma}$ a pro-$\mathtt{C}$-group of infinite rank. We consider non-trivial pro-$\mathtt{C}$-embedding problems
    \begin{equation} \label{eqn: embedding problem profinite}
    \xymatrix{
        \mathtt{\Gamma} \ar@{->>}^\beta[rd] \ar@{..>>}[d]\\
        \mathtt{G} \ar@{->>}^-\alpha[r] & \mathtt{H}
    }
    \end{equation}
    for $\mathtt{\Gamma}$. The following statements are equivalent:
    \begin{enumerate}
        \item Every pro-$\mathtt{C}$-embedding problem (\ref{eqn: embedding problem profinite}) with $\rank(\mathtt{H})<\rank(\mathtt{\Gamma})$ and $\rank(\mathtt{G})\leq \rank(\mathtt{\Gamma})$ has a solution.
        \item Every pro-$\mathtt{C}$-embedding problem (\ref{eqn: embedding problem profinite}) with $\rank(\mathtt{H})<\rank(\mathtt{\Gamma})$ and finite kernel has a solution.
        \item  Every pro-$\mathtt{C}$-embedding problem (\ref{eqn: embedding problem profinite}) with $\rank(\mathtt{H})<\rank(\mathtt{\Gamma})$ such that $\ker(\alpha)$ is a finite minimal normal subgroup of $\mathtt{G}$ has a solution.
        \item Every $\mathtt{C}$-embedding problem (\ref{eqn: embedding problem profinite}) has $\rank(\mathtt{\Gamma})$ solutions.
    \end{enumerate}
\end{theo}
To state our generalization of Theorem \ref{theo: saturation for profinite groups} we need two more definitions.

\begin{defi}
    Let $N\neq 1$ be a normal closed subgroup of a proalgebraic group $G$ such that $N$ is an algebraic group. If $N$ is finite, $N$ is an \emph{almost minimal} normal subgroup of $G$ if
    for every normal closed subgroup $N'$ of $G$ with $N'\subseteq N$ either $N'=N$ or $N'=1$.
    If $N$ is not finite, $N$ is an \emph{almost minimal} normal subgroup of $G$
    if for every normal closed subgroup $N'$ of $G$ with $N'\subseteq N$ either $N'=N$ or $N'$ is finite.

    An embedding problem (\ref{eqn: embedding problem defi}) has \emph{almost minimal kernel}, if its kernel is an almost minimal normal subgroup of~$G$.
\end{defi}

\begin{defi}
    A family $(\f_i)_{i\in I}$ of solutions of (\ref{eqn: embedding problem defi}) is called \emph{independent} if the induced morphism $\prod_{i\in I}\f_i\colon \Gamma\to \prod_{i\in I} (G\twoheadrightarrow H)$ is an epimorphism.
    % $\prod\f_i\colon \Gamma\to \prod (G\twoheadrightarrow H)$
\end{defi}

\begin{rem}
Since the coordinate ring of $\prod_{i\in I} (G\twoheadrightarrow H)$ is the union (direct limit) of tensor products $k[G]\otimes_{k[H]}\ldots\otimes_{k[H]} k[G]$, we see that the family $(\f_i)_{i\in I}$ is independent if and only if for every finite subset $\{i_1,\ldots,i_n\}$ of $I$ the morphism $k[G]\otimes_{k[H]}\ldots\otimes_{k[H]} k[G]\to k[\Gamma],\ a_1\otimes\ldots\otimes a_n\mapsto \f_{i_1}^*(a_1)\cdots\f^*_{i_n}(a_n)$ is injective. Thus, independence of the $\f_i$, corresponds to linear independence (over $k[H]=k[\Gamma/\ker(\beta)]$) of the $k[\Gamma/\ker(\f_i)]$ inside $k[\Gamma]$.
\end{rem}

\medskip

The main goal of this section is to prove the following theorem, which generalizes Theorem~\ref{theo: saturation for profinite groups} from finite groups to algebraic groups.

\begin{theo} \label{theo: saturated}
    Let $\CC$ be a formation and $\Gamma$ a pro-$\CC$-group of infinite rank. We consider non-trivial pro-$\CC$-embedding problems
        \begin{equation} \label{eqn: embedding problem thm}
    \xymatrix{
        \Gamma \ar@{->>}^\beta[rd] \ar@{..>>}[d]\\
        G \ar@{->>}^-\alpha[r] & H
    }
    \end{equation}
    for $\Gamma$. The following statements are equivalent:
    \begin{enumerate}
        \item Every pro-$\CC$-embedding problem (\ref{eqn: embedding problem thm}) with $\rank(H)<\rank(\Gamma)$ and $\rank(G)\leq \rank(\Gamma)$ has a solution.
        \item Every pro-$\CC$-embedding problem (\ref{eqn: embedding problem thm}) with $\rank(H)<\rank(\Gamma)$ and algebraic kernel has a solution.
        \item  Every pro-$\CC$-embedding problem (\ref{eqn: embedding problem thm}) with $\rank(H)<\rank(\Gamma)$ and almost minimal kernel has a solution.

        \item For every $\CC$-embedding problem (\ref{eqn: embedding problem thm}) and every normal closed subgroup $N$ of $\Gamma$ with $\rank(\Gamma/N)<\rank(\Gamma)$ and $N\subseteq\ker(\beta)$, there exists a solution $\f$ such that $\f(N)=\ker(\alpha)$.
%       \item For every $\CC$-embedding problem (\ref{eqn: embedding problem thm}) and every normal closed subgroup $N$ of $\Gamma$ with $\rank(\Gamma/N)<\rank(\Gamma)$, there exists a solution $\f$ such that $N/(\ker(\f)\cap N)\neq 1$ if $\ker(\alpha)$ is finite and $N/(\ker(\f)\cap N)$ is not finite if $\ker(\alpha)$ is not finite.
        \item For every $\CC$-embedding problem (\ref{eqn: embedding problem thm}) and every normal closed subgroup $N$ of $\Gamma$ with $\rank(\Gamma/N)<\rank(\Gamma)$ and $N\subseteq\ker(\beta)$, there exists a solution $\f$ such that $\f(N)\neq 1$ if $\ker(\alpha)$ is finite and $\dim(\f(N))>0$ if $\dim(\ker(\alpha))>0$.
        \item For every $\CC$-embedding problem (\ref{eqn: embedding problem thm}) we have $\rank(\Gamma/N)=\rank(\Gamma)$ if $\ker(\alpha)$ is finite and we have $\dim(\Gamma/N)=\rank(\Gamma)$ if $\dim(\ker(\alpha))>0$, where $N$ denotes the intersection of all kernels of all solutions to (\ref{eqn: embedding problem thm}).
        \item For every $\CC$-embedding problem (\ref{eqn: embedding problem thm}) there exist $\rank(\Gamma)$ independent solutions.
    \end{enumerate}
\end{theo}
%\enlargethispage{20mm}
The proof of Theorem \ref{theo: saturated} will be given at the end of this subsection. It consists of the following implications:

$$
\xymatrix{
    & {\rm (i)} \ar@{=>}^{\rm trivial}[d] \ar@/^7.0pc/@{=>}[ddddd] & \\
& {\rm (ii)} \ar@{=>}[rd] \ar@{=>}^-{\rm trivial}[ld] & \\
{\rm (iii)} \ar@{=>}^{{\rm Prop. \ref{prop: from minimal kernel to arbitrary}} }[ruu] & & {\rm (iv)} \ar@{=>}^-{\rm trivial}[ld] \\
& {\rm (v)}  \ar@{=>}^{\rm Prop. \ref{prop: from weak N condition to almost minimal}}[lu] & \\
& {\rm (vi)} \ar@{=>}[u] & \\
& {\rm (vii)} \ar@{=>}[u] &
}
 $$

 In \cite[Theorem~3.7]{BachmayrHarbaterHartmannWibmer:FreeDifferentialGaloisGroups} we translate the conditions from Theorem \ref{theo: saturated} to differential embedding problems. The reader may find it instructive to look at the meaning of the above conditions in this context.

 \begin{defi}
    Let $\CC$ be a formation. A pro-$\CC$-group of infinite rank satisfying the equivalent characterizations of Theorem \ref{theo: saturated} is called \emph{saturated}.
 \end{defi}

\begin{rem}
    A non-trivial pro-$\CC$-group $\Gamma$ of finite rank, i.e., a non-trivial $\CC$-group, cannot satisfy all conditions (i)-(vii) of Theorem \ref{theo: saturated}. For example, condition (i) (with $H=1$) implies that for any $G\in \CC$ and $n\geq 1$, the algebraic group $G^n$ is an epimorphic image of $\Gamma$. Thus the rank of $\Gamma$ must be infinite.
\end{rem}

It turns out that in characteristic zero, saturated pro-$\CC$-groups with rank not less than $|k|$ are free pro-$\CC$-groups in disguise (Theorem \ref{theo: free=saturated}). The following remark explains why we chose the expression ``saturated''.

 \begin{rem}
    Let $\mathtt{C}$ be a formation of finite groups. There exists a certain first order theory $T$ such that models of $T$ correspond to the class of profinite groups having a certain property (called the Iwasawa property in \cite{Chatzidakis:ModelTheoryOfProfiniteGroupsHavingTheIwasawProperty} and \cite{CherlinVanDenDriesMacinyre} and the embedding property in \cite{FriedJarden:FieldArithmetic}). Moreover, a pro-$\mathtt{C}$-group satisfies the equivalent conditions of Theorem \ref{theo: saturation for profinite groups} if and only if it corresponds to a saturated model of $T$. See \cite{Chatzidakis:ModelTheoryOfProfiniteGroupsHavingTheIwasawProperty} and \cite{CherlinVanDenDriesMacinyre}.
    It seems conceivable that there is some first order theory $T$ such that saturated pro-$\CC$-groups correspond to saturated models of $T$. In any case, the abelian groups $M$, such that $D(M)$ is a saturated pro-diagonalizable group, are exactly the saturated models of a certain first order theory (namely, the theory of divisible abelian groups with infinite $p$-torsion for every prime $p$). See Example \ref{ex: saturated prodiagonalizable group}.
 \end{rem}

\begin{rem} \label{rem: equivalent conditions for f(N)=ker(alpha)}
    Let (\ref{eqn: embedding problem thm}) be a pro-$\CC$-embedding problem for $\Gamma$ and $N$ a normal closed subgroup of $\Gamma$ with $N\subseteq \ker(\beta)$. If $\f$ is a solution to (\ref{eqn: embedding problem thm}), then $\f(N)\subseteq \ker(\alpha)$ and $N\ker(\f)\subseteq\ker(\beta)$. Moreover, for a solution $\f$, the following statements are equivalent:
    \begin{enumerate}
        \item The induced map $\Gamma\to G\times_H\Gamma/N$ is an epimorphism.
        \item $\f(N)=\ker(\alpha)$.
        \item $N\ker(\f)=\ker(\beta)$.
    \end{enumerate}
\end{rem}
\begin{proof}
(i)$\Rightarrow$(ii):
%$$
%\xymatrix{
%   \Gamma \ar@{->>}[rd] \ar@{..>>}[d]\\
%   G\times_H \Gamma/N \ar@{->>}[r] & \Gamma/N
%}
%$$
%has a solution $\f\colon\Gamma\twoheadrightarrow   G\times_H \Gamma/N$. Let $\f=\pi_G\circ\f'$, where $\pi_G\colon     G\times_H \Gamma/N\to G$ denotes the projection onto the first factor. Then $\f$ is a solution to (\ref{eqn: embedding problem thm}).
Let $R$ be a $k$-algebra and $g\in\ker(\alpha)(R)\leq G(R)$. Then $$(g,1)\in (G\times_H \Gamma/N)(R)=G(R)\times_{H(R)} (\Gamma/N)(R).$$
Since  $\Gamma\to G\times_H\Gamma/N$ is an epimorphism, there exists a faithfully flat $R$-algebra $S$ and a $\gamma\in \Gamma(S)$ such that $(\f(\gamma),\overline{\gamma})=(g,1)$. But then $\gamma\in N(S)$ and $\f(\gamma)=g$. So $\f(N)=\ker(\alpha)$ as desired.

(ii)$\Rightarrow$(i): Let $R$ be a $k$-algebra and $(g,h)\in (G\times_H\Gamma/N)(R)$. Then there exists a faithfully flat $R$-algebra $S$ and $\gamma\in \Gamma(S)$ such that $\gamma$ maps to $h$, i.e., $h=\overline{\gamma}$. We have $\alpha(g)=\beta(\gamma)$ and so $\alpha(\f(\gamma))=\beta(\gamma)=\alpha(g)$. Therefore $\f(\gamma)^{-1}g\in\ker(\alpha)(S)$.  As $\ker(\alpha)=\f(N)$, there exists a faithfully flat $S$-algebra $S'$ and $n\in N(S')$ with $\f(\gamma)^{-1}g=\f(n)$. But then $\gamma n$ maps to $(g,h)$ under $\Gamma\to G\times_H\Gamma/N$.

(iii)$\Rightarrow$(ii): If $N\ker(\f)=\ker(\beta)$, then $\f(N)=\f(N\ker(\f))=\f(\ker(\beta))=\ker(\alpha)$.

(ii)$\Rightarrow$(iii): If $\f(N)=\ker(\alpha)$, then $\f(N\ker(\f))=\f(N)=\ker(\alpha)$ and also $\f(\ker(\beta))=\ker(\alpha)$. Since both, $N\ker(\f)$ and $\ker(\beta)$, contain $\ker(\f)$, this implies  $N\ker(\f)=\ker(\beta)$.
\end{proof}

\begin{rem} While there is a clear correspondence between (i), (ii) and (iii) of Theorem \ref{theo: saturation for profinite groups} and Theorem \ref{theo: saturated}, the other characterizations in Theorem \ref{theo: saturated} (with maybe the exception of (vii)) do not look like (iv) of Theorem~\ref{theo: saturation for profinite groups} on the face of it. However, it is not hard to see that (vi) of Theorem \ref{theo: saturated} corresponds to (iv) of Theorem \ref{theo: saturation for profinite groups} in the case when $\CC=\mathtt{C}_k$ is a formation of finite constant algebraic groups obtained from a formation $\mathtt{C}$ of finite groups.

    Indeed, we will show that (vi) of Theorem \ref{theo: saturated} holds for a $\mathtt{C}_k$-embedding problem (\ref{eqn: embedding problem thm}) if and only if it has $\rank(\Gamma)$ solutions:

    We first note that two solutions of (\ref{eqn: embedding problem thm}) that have the same kernel, only differ by an automorphism of $G$. Since there are only finitely many such automorphisms, the number of kernels of solutions equals $\rank(\Gamma)$ if and only if there are $\rank(\Gamma)$ solutions. Moreover, if $N$ is the intersection of all solutions to (\ref{eqn: embedding problem thm}), then the finite intersections of kernels of solutions yield a neighborhood basis at $1$ for $\Gamma/N$. Now in an infinite profinite group the cardinality of a neighborhood basis at $1$ consisting of normal open ($\simeq$coalgebraic) subgroups equals the cardinality of all open subgroups (since there are only finitely many subgroups containing a given open normal subgroup). It follows that $\rank(\Gamma/N)=\rank(\Gamma)$ if and only if (\ref{eqn: embedding problem thm}) has $\rank(\Gamma)$ solutions.

%
%
%       number of distinct solutions of (\ref{eqn: embedding problem thm}) agrees with the cardinality of the set of all kernels of all solutions of (\ref{eqn: embedding problem thm}).
%
%   Assume that (\ref{eqn: embedding problem thm}) has $\rank(\Gamma)$ solutions. Suppose for a contradiction that (vi) does not hold, then $N$ is contained in $\ker(\f)$ for every solution $\f$. Thus for every solution $\f$, we obtain a coalgebraic subgroup $\ker(\f)/N$ of $\Gamma/N$. Since $G$ is finite, a counting argument  shows that the set of kernels of all solutions has the same cardinality as the set of all solutions. But then $\rank(\Gamma/N)\geq\rank(\Gamma)$ since there are $\rank(\Gamma)$ solutions; a contradiction.
%
%   Conversely, assume that (vi) holds for (\ref{eqn: embedding problem thm}). Let $N$ be the intersection of all kernels of solutions. Suppose $\rank(\Gamma/N)<\rank(\Gamma)$. Then (vi) implies the existence of a solution with $N\nsubseteq\ker(\f)$; a contradiction. Thus $\rank(\Gamma/N)=\rank(\Gamma)$ and therefore $\rank(N)=\rank(\Gamma)$. Since the kernels of a solutions are a basis  there must be $\rank(\Gamma)$ solutions.
\end{rem}

The following three lemmas are a preparation for the proof of (iii)$\Rightarrow$(i) in Theorem \ref{theo: saturated}. For a finite affine scheme $X$ over $k$, let us write $|X|$ for the vector space dimension $\dim_k(k[X])$ of the coordinate ring of $X$, i.e., the number of points of $X$ counted with multiplicities. For clarity of exposition we separated the following lemma from the proof of Lemma \ref{lemma: Hoelder type for algebraic}.
\begin{lemma} \label{lemma: bound size}
Let $G$ be an algebraic group. Then $|G/N|$ is bounded, as $N$ varies over all normal closed subgroups of $G$ with $\dim(N)=\dim(G)$.
\end{lemma}
\begin{proof}
In characteristic zero $|G/G^o|$ is a bound. In positive characteristic, if $N\subseteq G$ with $\dim(N)=\dim(G)$, we can only conclude that $(G^o)_{\operatorname{red}}$ is contained in $N$. So $|G/N|= |G_{\overline{k}}/N_{\overline{k}}|\leq |G_{\overline{k}}/(G_{\overline{k}}^o)_{\operatorname{red}}|$. Here we passed to $\overline{k}$ to guarantee that the underlying reduced subscheme of a subgroup is a subgroup. Also note that $G_{\overline{k}}/(G_{\overline{k}}^o)_{\operatorname{red}}$ is a zero dimensional algebraic scheme and therefore is finite.
\end{proof}

\begin{lemma} \label{lemma: Hoelder type for algebraic}
    Let $G$ be a proalgebraic group and $N, \widetilde{N}$ normal closed subgroups of $G$ with $\widetilde{N}\subseteq N$ and such that $N/\widetilde{N}$ is algebraic. Then there exists a finite chain
    $$ N=N_0\geq N_1\geq \ldots\geq N_n=\widetilde{N}$$ of normal closed subgroups of $G$ such that $N_i/N_{i+1}$ is an almost minimal normal subgroup of $G/N_{i+1}$ for $i=0,\ldots,n-1$.
\end{lemma}
\begin{proof}
    Assume first that $N/\widetilde{N}$ is finite. If there is no normal closed subgroup $N'$ of $G$ that strictly contains $\widetilde{N}$ and is strictly contained in $N$, then $N/\widetilde{N}$ is an almost minimal normal subgroup of $G/\widetilde{N}$ and we can choose $n=1$. So let us assume that such an $N'$ exists. In this case, we may choose an $N'$ such that $|N'/\widetilde{N}|$ is maximal (i.e., $|N/N'|$ is minimal) and call it $N_1$.
    %   Among all normal closed subgroups $N'$ of $G$ that contain $\widetilde{N}$ and are strictly contained in $N$, choose one such that $|N'|$ is maximal and call it $N_1$.
%
    Then $N/N_1$ is an almost minimal normal subgroup of $G/N_1$. Since $|N_1/\widetilde{N}|<|N/\widetilde{N}|$ we may conclude by induction on $|N/\widetilde{N}|$.

    Let us now consider the general case. We proceed by induction on $\dim(N/\widetilde{N})$. The base case $\dim(N/\widetilde{N})=0$ was just handled above. So we may assume that $\dim(N/\widetilde{N})>0$. As above, if there is no normal closed subgroup of $G$ that strictly contains $\widetilde{N}$ and is strictly contained in $N$ we may choose $n=1$.
    Otherwise, among all normal closed subgroups $N'$ of $G$ that are strictly contained in $N$ and strictly contain $\widetilde{N}$, choose one such that $\dim(N'/\widetilde{N})$ is maximal (i.e., $\dim(N/N')$ minimal) and call it $N_1$. If $\dim(N/N_1)>0$, then $N/N_1$ is not finite and an almost minimal normal subgroup of $G/N_1$, because a normal closed subgroup of $G/N_1$ that is strictly contained in $N/N_1$ must have dimension zero and is therefore a finite subgroup. As $\dim(N_1/\widetilde{N})<\dim(N/\widetilde{N})$ we conclude by induction on $\dim(N/\widetilde{N})$.

    If $\dim(N/N_1)=0$ we proceed as follows: Since $|N/N'|$ is bounded, as $N'$ varies over all normal closed subgroups of $G$ with $N\supseteq N'\supseteq\widetilde{N}$ and $\dim(N/N')=0$ (Lemma \ref{lemma: bound size} applied to $G=N/\widetilde{N}$), we may choose one such that $|N/N'|$ is maximal. Let us call this group $\overline{N}$. Since $N/\overline{N}$ is finite, we know from the first part of this proof that there exists a sequence
        $$N=N_0\geq N_1\geq \ldots\geq N_m=\overline{N}$$
    as in the statement of the lemma. Now if $N'$ is a normal closed subgroup of $G$ strictly contained in $\overline{N}$ and containing $\widetilde{N}$, then $\dim(N')<\dim(\overline{N})$ since otherwise the maximality of $|N/\overline{N}|$ would be contradicted. Considering the normal closed subgroups $\overline{N}\supseteq\widetilde{N}$ of $G$ we may conclude by induction since this case was already treated above.
\end{proof}

It is not possible to strengthen ``almost minimal'' to ``minimal'' in the above lemma. For example, consider $G=\Gm=N$ and $\widetilde{N}=1$.

\begin{lemma} \label{lemma: prepare for minimal prop}
    Let $G$ be a proalgebraic group and $N$ a normal closed subgroup. Then there exists an ordinal number $\mu$ and a chain $$N=N_0\geq N_1\geq \ldots \geq N_\lambda\geq \ldots \geq N_\mu=1$$
    of normal closed subgroups $N_\lambda$ ($\lambda\leq \mu$) of $G$ such that
    \begin{enumerate}
        \item $N_{\lambda+1}=N_{\lambda}$  or $N_\lambda/N_{\lambda+1}$ is an almost minimal normal closed subgroup of $G/N_{\lambda+1}$ for $\lambda<\mu$,
        \item $N_\lambda=\bigcap_{\lambda'<\lambda}N_{\lambda'}$ if $\lambda$ is a limit ordinal and
        \item if $G$ has infinite rank and $\rank(G/N)<\rank(G)$, then $\rank(G/N_\lambda)<\rank(G)$ for $\lambda<\mu$.
    \end{enumerate}
\end{lemma}
\begin{proof}
If $N$ is algebraic, the claim follows from Lemma \ref{lemma: Hoelder type for algebraic}. So let us assume that $N$ is not algebraic. Let $\mathcal{N}$ denote a neighborhood basis at $1$ of $G$ with $|\mathcal{N}|=\rank(G)$. According to Corollary \ref{cor: neighborhood basis for subgroup} the set $\mathcal{N}(N)=\{N\cap N'|\ N'\in\mathcal{N}\}$ is a neighborhood basis at $1$ for $N$. Let $\mu$ be the smallest ordinal number such that $|\mu|=|\mathcal{N}(N)|$. So we may write $\mathcal{N}(N)=\{U_\lambda|\ \lambda<\mu\}$. We set $N_0=N$ and for $1\leq\lambda\leq\mu$ we set $N_\lambda=\bigcap_{\lambda'<\lambda} U_\lambda$. Clearly the $N_\lambda$'s form a descending chain of normal closed subgroups of $G$ such that $N_0=N$ and $N_\mu=1$. Moreover $N_\lambda/N_{\lambda+1}=N_\lambda/(N_\lambda\cap U_{\lambda+1})=U_{\lambda+1}N_\lambda/U_{\lambda+1}$ is algebraic, because $U_{\lambda+1}N_\lambda/U_{\lambda+1}$ embeds into $N/U_{\lambda+1}$.
If $N_\lambda/N_{\lambda+1}$ is not an almost minimal normal closed subgroup of $G/N_{\lambda+1}$, we can use Lemma \ref{lemma: Hoelder type for algebraic} to insert finitely many subgroups between $N_\lambda$ and $N_{\lambda+1}$ to achieve (i). On the other hand, (ii) is satisfied by construction.

If $G$ is algebraic statement (iii) is void. So let us assume that $G$ is not algebraic, i.e., $\rank(G)$ is infinite. By Proposition \ref{prop: bound rank} (applied to $G/N$)
 there exists a chain $$G=M_0\geq M_1\geq \ldots \geq M_\xi\geq \ldots \geq M_\nu=N$$ of normal closed subgroups $M_\xi$ of $G$ such that
 \begin{enumerate}
    \item $M_\xi/M_{\xi+1}$ algebraic for all $\xi<\nu$,
    \item $N_\xi=\bigcap_{\xi'<\xi}N_{\xi'}$ if $\xi$ is a limit ordinal
 \end{enumerate}
and $|\nu|=\rank(G/N)$. Assume $\rank(G/N)<\rank(G)$ and fix $\lambda<\mu$. Proposition \ref{prop: bound rank} applied to the chain
$$M_0\geq M_1\geq \ldots \geq M_\nu=N=N_0\geq N_1\geq \ldots \geq N_\lambda$$ shows that $\rank(G/N_\lambda)\leq|\nu+\lambda|=|\nu|+|\lambda|$. But $|\nu|=\rank(G/N)<\rank(G)$ and $|\lambda|<|\mu|=\rank(G)$ by choice of $\mu$.
\end{proof}

We are now prepared to prove the implication (iii)$\Rightarrow$(i) in Theorem \ref{theo: saturated}.
\begin{prop} \label{prop: from minimal kernel to arbitrary}
    Let $\CC$ be a formation and $\Gamma$ a pro-$\CC$-group of infinite rank. We consider pro-$\CC$-embedding problems
    \begin{equation} \label{eqn: embedding problem prop minimal}
    \xymatrix{
        \Gamma \ar@{->>}[rd] \ar@{..>>}[d]\\
        G \ar@{->>}[r] & H
    }
    \end{equation}
    for $\Gamma$. If every pro-$\CC$-embedding problem with $\rank(H)<\rank(\Gamma)$ and almost minimal algebraic kernel has a solution, then every pro-$\CC$-embedding problem with $\rank(H)<\rank(\Gamma)$ and $\rank(G)\leq \rank(\Gamma)$ has a solution.
\end{prop}
\begin{proof}
    Assume that an embedding problem (\ref{eqn: embedding problem prop minimal}) with $\rank(H)<\rank(\Gamma)$ and $\rank(G)\leq\rank(\Gamma)$ is given. Let $N$ be the kernel of this embedding problem.   According to Lemma \ref{lemma: prepare for minimal prop} there exists an ordinal number $\mu$ and a chain $$N=N_0\geq N_1\geq \ldots \geq N_\lambda\geq \ldots \geq N_\mu=1$$
    of normal closed subgroups $N_\lambda$ ($\lambda\leq \mu$) of $G$ such that
    \begin{enumerate}
        \item $N_{\lambda+1}=N_{\lambda}$  or $N_\lambda/N_{\lambda+1}$ is an almost minimal normal closed subgroup of $G/N_{\lambda+1}$ for $\lambda<\mu$ and
        \item $N_\lambda=\bigcap_{\lambda'<\lambda}N_{\lambda'}$ if $\lambda$ is a limit ordinal and
        \item if $G$ has infinite rank and $\rank(H)<\rank(G)$, then $\rank(G/N_\lambda)<\rank(G)$ for $\lambda<\mu$.
    \end{enumerate}
We will show by (transfinite) induction that for each $\lambda\leq\mu$ there exists an epimorphism $\f_\lambda\colon\Gamma\twoheadrightarrow G/N_\lambda$ such that the diagram
    $$\xymatrix{
    \Gamma \ar@{->>}^{\f_{\lambda'}}[rd] \ar@{->>}_-{\f_{\lambda}}[d]\\
    G/N_{\lambda} \ar@{->>}[r] &  G/N_{\lambda'}
}$$
commutes for $\lambda'\leq\lambda$. We let $\f_0\colon\Gamma\twoheadrightarrow G/N_0=H$ be the epimorphism of the given embedding problem. Assume now that $\f_{\lambda'}$ has been constructed for each $\lambda'<\lambda$. We have to distinguish two cases:
If $\lambda$ is a limit ordinal, then $N_\lambda=\bigcap_{\lambda'<\lambda}N_{\lambda'}$. So $G/N_\lambda=\varprojlim_{\lambda'} G/N_{\lambda'}$ and we set $\f_\lambda=\varprojlim_{\lambda'}\f_{\lambda'}$.

If $\lambda$ is a successor ordinal, say $\lambda=\nu+1$, we would like to define $\f_\lambda$ as a solution to the embedding problem
    \begin{equation} \label{eqn: intermediate embedding problem}
    \xymatrix{
    \Gamma \ar@{->>}^{\f_{\nu}}[rd] \ar@{..>>}[d]\\
    G/N_{\lambda} \ar@{->>}[r] &  G/N_{\nu}
}.
    \end{equation}
If $N_{\lambda}=N_{\nu}$, we may choose $\f_{\lambda}=\f_\nu$. Otherwise, (\ref{eqn: intermediate embedding problem}) has an almost minimal algebraic kernel $N_\nu/N_{\nu +1}$. So to establish the existence of $\f_\lambda$, it suffices to verify that $\rank(G/N_\nu)<\rank(\Gamma)$. If $\rank(G)$ is finite, this follows from the assumption that $\rank(\Gamma)$ is infinite. So we may assume that $\rank(G)$ is infinite. If $\rank(H)=\rank(G)$, then $\rank(G/N_\nu)\leq\rank(G)=\rank(H)<\rank(\Gamma)$. If $\rank(H)<\rank(G)$, then $\rank(G/N_{\nu})<\rank(G)\leq\rank(\Gamma)$ by (iii). By construction, the epimorphism $\f_\mu\colon \Gamma\twoheadrightarrow G$ is a solution to (\ref{eqn: embedding problem prop minimal}).
\end{proof}

%\begin{defi}
%   Let $\Gamma$ be a pro-$\CC$-group and
%       $$\xymatrix{
%       \Gamma \ar@{->>}[rd] \ar@{..>>}[d]\\
%       G \ar@{->>}[r] & H
%   }
%$$ a $\CC$-embedding problem for $\Gamma$.
%\begin{itemize}
%   \item If the kernel is trivial, i.e., the embedding problem is trivial, it is \emph{maximally solvable}.
%   \item If the kernel is non-trivial but finite, the embedding problem is \emph{maximally solvable} if for every normal closed subgroup $N$ of $\Gamma$ with $\rank(\Gamma/N)<\rank(\Gamma)$ there exists a solution $\f$ such that $N/(N\cap\ker(\f))\neq 1$, i.e., $N\nsubseteq \ker(\f) $.
%\item  If the kernel is not finite, the embedding problem is \emph{maximally solvable} if for every normal closed subgroup $N$ of $\Gamma$ with $\rank(\Gamma/N)<\rank(\Gamma)$ there exists a solution $\f$ such that $N/(N\cap\ker(\f))$, is not finite.
%\end{itemize}
%\end{defi}
%
%

%\subsubsection{Fibre products} \label{subsec: fibre products}

We now work towards a proof of (v)$\Rightarrow$(iii) in Theorem \ref{theo: saturated}.

%If $G_1\to H$ and $G_2\to H$ are morphisms of proalgebraic groups, then the fibre product $G_1\times_H G_2$ is naturally a proalgebraic group and the diagram
%$$
%\xymatrix{
%   & G_1\times_H G_2 \ar[ld] \ar[rd] & \\
%   G_1 \ar[rd] & & G_2 \ar[ld] \\
%   & H &
%}
%$$
%satisfies the universal property in the category of proalgebraic groups. A diagram as above will be called a \emph{cartesian square}.
%
%
%If $N_1,N_2$ are normal closed subgroups of a proalgebraic group $G$ it follows from the universal properties that $G/N_1\times_{G/N_1N_2} G/N_2\simeq G/(N_1\cap N_2)$. In particular, if $N_1\cap N_2=1$, then
%$$
%\xymatrix{
%   & G \ar[ld] \ar[rd] & \\
%   G/N_1 \ar[rd] & & G/N_2 \ar[ld] \\
%   & G/N_1N_2 &
%}
%$$
%is a cartesian square.

\begin{lemma} \label{lemma: every morphism is fibre product}
    Let $\alpha\colon G\twoheadrightarrow H$ be an epimorphism of proalgebraic groups such that $\ker(\alpha)$ is algebraic. Then there exist algebraic groups $H'$ and $H''$ and epimorphisms $H\twoheadrightarrow H''$, $H'\twoheadrightarrow H''$ such that $G\simeq H'\times_{H''}H$ and $\alpha$ can be identified with the projection onto the second factor. Moreover, if $\ker(\alpha)$ is an almost minimal normal subgroup of $G$, then the kernel of $H'\to H''$ is an almost minimal normal subgroup of $H'$.
\end{lemma}
\begin{proof}
    If $N_1,N_2$ are normal closed subgroups of a proalgebraic group $G$, it follows from the universal properties that $G/N_1\times_{G/N_1N_2} G/N_2\simeq G/(N_1\cap N_2)$. In particular, $G/N_1\times_{G/N_1N_2} G/N_2\simeq G$ if $N_1\cap N_2=1$.

    According to Corollary \ref{cor: intersection 1}, there exists a coalgebraic subgroup $N$ of $G$ with $N\cap\ker(\alpha)=1$.
    So $G\simeq G/N\times_{G/N\ker(\alpha)} G/\ker(\alpha)=H'\times_{H''}H$.

    Note that $\ker(\alpha)$ is isomorphic to $\ker(H'\to H'')$. If $N'$ is a normal closed subgroup of $H'$ contained in $\ker(H'\to H'')$. Then $N'\simeq N'\times_{H''} 1$ corresponds to a normal closed subgroup of $G$ contained in $\ker(\alpha)$.
\end{proof}

The following proposition proves (v)$\Rightarrow$(iii) of Theorem \ref{theo: saturated}.

\begin{prop} \label{prop: from weak N condition to almost minimal}
    Let $\CC$ be a formation and $\Gamma$ a pro-$\CC$-group of infinite rank. We consider non-trivial pro-$\CC$-embedding problems
    \begin{equation} \label{eqn: embedding problem prop hardest}
    \xymatrix{
        \Gamma \ar@{->>}^{\beta}[rd] \ar@{..>>}[d]\\
        G \ar@{->>}^-{\alpha}[r] & H
    }
    \end{equation}
    for $\Gamma$. Assume that for every $\CC$-embedding problem (\ref{eqn: embedding problem prop hardest}) and every normal closed subgroup $N$ of $\Gamma$ with $N\subseteq\ker(\beta)$ and $\rank(\Gamma/N)<\rank(\Gamma)$, there exists a solution $\f$ such that $\f(N)\neq 1$ if $\ker(\alpha)$ is finite and $\dim(\f(N))>0$ if $\dim(\ker(\alpha))>0$.
    Then every pro-$\CC$-embedding problem (\ref{eqn: embedding problem prop hardest}) with $\rank(H)<\rank(\Gamma)$ and almost minimal kernel has a solution.
\end{prop}
\begin{proof}
    Assume that a pro-$\CC$-embedding problem (\ref{eqn: embedding problem prop hardest}) with $\rank(H)<\rank(\Gamma)$ and almost minimal kernel is given. According to Lemma \ref{lemma: every morphism is fibre product}, the epimorphism $G\twoheadrightarrow H$ can be written as $H'\times_{H''} H\to H$, where $H'\twoheadrightarrow H''$ is an epimorphism of algebraic groups with almost minimal kernel and $H\twoheadrightarrow H''$. Let $N$ be the kernel of $\beta\colon\Gamma\twoheadrightarrow H$. Then $\rank(\Gamma/N)=\rank(H)<\rank(\Gamma)$. Since $N$ lies in the kernel of $\Gamma\to H''$, it follows from the assumption, that the $\CC$-embedding problem
    \begin{equation} \label{eqn: embedding problem for shortening proof}
        \xymatrix{
        \Gamma \ar@{->>}[rd] \ar@{..>>}[d]\\
        H' \ar@{->>}[r] & H''
    }
    \end{equation}
    has a solution $\f\colon \Gamma\twoheadrightarrow H'$ such that
    \begin{itemize}
        \item if $\ker(H'\to H'')$ is finite, then $\f(N)\neq 1$ and %$N/(N\cap\ker(\f))\neq 1$ and
        \item if $\dim(\ker(H'\to H''))>0$, then $\dim(\f(N))>0$. %$N/(N\cap\ker(\f))$ is not finite.
    \end{itemize}
Clearly, the induced map $\psi=(\f,\beta)\colon \Gamma\to H'\times_{H''} H$ is a weak solution of the given embedding problem (\ref{eqn: embedding problem prop hardest}). We will show that $\psi$ actually is a solution, i.e., an epimorphism. Since $N$ lies in the kernel of $\Gamma\twoheadrightarrow H''$ and $\f$ is a solution of (\ref{eqn: embedding problem for shortening proof}) we have $\f(N)\subseteq \ker(H'\to H'')$. As $\f\colon\Gamma\to H'$ is an epimorphism, $\f(N)$ is a normal closed subgroup of $H'$. Since $\ker(H'\to H'')$ is almost minimal, we see that $\f(N)=\ker(H'\to H'')$. It thus follows from Remark \ref{rem: equivalent conditions for f(N)=ker(alpha)} that $\psi\colon \Gamma\to H'\times_{H''}H=G$ is an epimorphism. So $\psi$ is the desired solution of (\ref{eqn: embedding problem prop hardest}).
\end{proof}

Finally, we can put all pieces together for the proof of  Theorem \ref{theo: saturated}:

\vspace{5mm}

\noindent \emph{Proof of Theorem \ref{theo: saturated}:}  Clearly, (i)$\Rightarrow$(ii)$\Rightarrow$(iii). By Proposition \ref{prop: from minimal kernel to arbitrary} (iii)$\Rightarrow$(i). So (i), (ii) and (iii) are equivalent.

Let us show (ii)$\Rightarrow$(iv): So let (\ref{eqn: embedding problem thm}) be a $\CC$-embedding problem and $N$ a normal closed subgroup of $\Gamma$ with $N\subseteq\ker(\beta)$ and $\rank(\Gamma/N)<\rank(\Gamma)$. By assumption, the embedding problem
$$
    \xymatrix{
    \Gamma \ar@{->>}[rd] \ar@{..>>}[d]\\
    G\times_H \Gamma/N \ar@{->>}[r] & \Gamma/N
}
$$
has a solution $\f'\colon\Gamma\twoheadrightarrow   G\times_H \Gamma/N$. Let $\f=\pi_G\circ\f'$, where $\pi_G\colon     G\times_H \Gamma/N\to G$ denotes the projection onto the first factor. Then $\f$ is a solution to (\ref{eqn: embedding problem thm}). According to Remark \ref{rem: equivalent conditions for f(N)=ker(alpha)} we have $\f(N)=\ker(\alpha)$.

%Let $R$ be a $k$-algebra and $g\in\ker(\alpha)(R)\leq G(R)$. Then $$(g,1)\in (G\times_H \Gamma/N)(R)=G(R)\times_{H(R)} (\Gamma/N)(R).$$
%
% Since $\f'$ is an epimorphism, there exists a faithfully flat $R$-algebra $S$ and a $\gamma\in \Gamma(S)$ mapping to $(g,1)$ under $\f'$ (Remark ). But then $\gamma\in N(S)$ and $\f(\gamma)=g$. So $\f(N)=\ker(\alpha)$ as desired.
%
Now (iv)$\Rightarrow$(v) is obvious and (v)$\Rightarrow$(iii) by Proposition \ref{prop: from weak N condition to almost minimal}. So we already know that the first five conditions of Theorem \ref{theo: saturated} are equivalent.

Let us show (vi)$\Rightarrow$(v): So we assume that a $\CC$-embedding problem (\ref{eqn: embedding problem thm}) is given, together with a normal closed subgroup of $N$ of $\Gamma$ with $N\subseteq\ker(\beta)$ and $\rank(\Gamma/N)<\rank(\Gamma)$. Let $N_{\alpha,\beta}$ denote the intersection of all kernels of all solutions to (\ref{eqn: embedding problem thm}).

 Let us first assume that $\ker(\alpha)$ is finite. By assumption we have $\rank(\Gamma/N_{\alpha,\beta})=\rank(\Gamma)$. On the other hand, $\rank(\Gamma/N_{\alpha,\beta}N)\leq\rank(\Gamma/N)<\rank(\Gamma)$. This shows that the kernel $N_{\alpha,\beta}N/N_{\alpha,\beta}=N/(N\cap N_{\alpha,\beta})$ of the epimorphism $\Gamma/N_{\alpha,\beta}\twoheadrightarrow\Gamma/N_{\alpha,\beta}N$ must be non-trivial. In other words, $N$ is not contained in $N_{\alpha,\beta}$. Thus there exists a solution $\f$ of (\ref{eqn: embedding problem thm}) such that $N$ is not contained in $\ker(\f)$, i.e., $\f(N)\neq 1$ as desired.

Let us now assume that $\dim(\ker(\alpha))>0$. By assumption, $\dim(\Gamma/N_{\alpha,\beta})=\rank(\Gamma)$. On the other hand, $\dim(\Gamma/N_{\alpha,\beta}N)\leq\dim(\Gamma/N)\leq\rank(\Gamma/N)<\rank(\Gamma)$, using Lemmas \ref{lemma: dim for subgroups and quotients} and \ref{lemma: dim bounded by rank}. Applying Lemma \ref{lemma: dim for kernel zero dimensional} to the epimorphism $\Gamma/N_{\alpha,\beta}\twoheadrightarrow\Gamma/N_{\alpha,\beta}N$ shows that its kernel $N_{\alpha,\beta}N/N_{\alpha,\beta}=N/(N\cap N_{\alpha,\beta})$ has positive dimension. Using Lemma \ref{lemma: dim for intersection} we see that $N/(N\cap\ker(\f))$ has positive dimension for some solution $\f$ of (\ref{eqn: embedding problem thm}). Thus $\f(N)\simeq N/(N\cap\ker(\f))$ has positive dimension as desired.

We next show that (i)$\Rightarrow$(vii): So we assume that a $\CC$-embedding problem (\ref{eqn: embedding problem thm}) is given. Let $I$ be a set of cardinality $\rank(\Gamma)$ and set $G'=\prod_{I}(G\twoheadrightarrow H)$. Then $\rank(G')\leq\rank(\Gamma)$. By assumption the embedding problem
$$
\xymatrix{
    \Gamma \ar@{->>}[rd] \ar@{..>>}[d]\\
    G' \ar@{->>}[r] & H
}
$$
has a solution $\f\colon \Gamma\twoheadrightarrow G'$. For every $i\in I$ we can compose $\f$ with the projection $\pi_i\colon G'\to G$ that picks out the $i$-th component, to obtain a solution $\f_i\colon\Gamma\twoheadrightarrow G$ of (\ref{eqn: embedding problem thm}).
By construction these solutions are independent.

%
% Then $\ker(\f)=\bigcap_{i\in I}\ker(\f_i)$.
%
%Let $N$ denote the intersection of all kernels of solutions to  (\ref{eqn: embedding problem thm}). Then $N\subseteq \bigcap_{i\in I}\ker(\f_i)$. So we have an epimorphism $\Gamma/N\twoheadrightarrow G/\ker(\f)=G'$. It follows that $\rank(\Gamma/N)\geq\rank(G')=\rank(\Gamma)$. So $\rank(\Gamma/N)=\rank(\Gamma)$ as desired.
%
%Similarly, $\dim(\Gamma/N)\geq \dim(G')\geq\rank(\Gamma)$ by Lemmas \ref{lemma: dim for subgroups and quotients} and \ref{lemma: dim of fibre product}, where, for the last inequality, we assume that $\dim(\ker(\alpha))>0$. On the other hand, $\rank(\Gamma)\geq\dim(\Gamma)\geq\dim(\Gamma/N)$. So $\dim(\Gamma/N)=\rank(\Gamma)$ as desired.
%
%
%At this point we have shown that the first six conditions of Theorem \ref{theo: saturated} are equivalent. To show (i)$\Rightarrow$(vii) we can proceed as in (i)$\Rightarrow$(vi) and note that the solutions $\f_i$ are independent by construction.

Finally, we show (vii)$\Rightarrow$(vi): So we assume that a $\CC$-embedding problem (\ref{eqn: embedding problem thm}) is given together with independent solutions $(\f_i)_{i\in I}$, where $|I|=\rank(\Gamma)$. The intersection $N$ of all kernels of all solution to $(\ref{eqn: embedding problem thm})$ is contained in $\bigcap_{i\in I}\ker(\f_i)$. So the isomorphism $\Gamma/\bigcap_{i\in I}\ker(\f_i)\to\prod_{i\in I}(G\twoheadrightarrow H)$ yields an epimorphism $\Gamma/N\twoheadrightarrow \prod_{i\in I}(G\twoheadrightarrow H)$. By Lemma \ref{lemma: rank for subgroups and quotients} we have
$$\rank(\Gamma/N)\geq \rank(\prod_{i\in I}(G\twoheadrightarrow H))\geq\rank(\prod_{i\in I}\ker(\alpha))=|I|=\rank(\Gamma).$$
Moreover, if $\dim(\ker(\alpha))>0$ we have $\dim(\Gamma/N)\geq\dim(\prod_{i\in I}(G\twoheadrightarrow H))\geq|I|=\rank(\Gamma)$ by Lemmas~\ref{lemma: dim for subgroups and quotients} and \ref{lemma: dim of fibre product}. But then necessarily $\dim(\Gamma/N)=\rank(\Gamma)$, because $\dim(\Gamma/N)\leq \dim(\Gamma)\leq\rank(\Gamma)$ (Lemmas \ref{lemma: dim for subgroups and quotients} and~\ref{lemma: dim bounded by rank}).
\qed

\subsection{Existence of saturated pro-$\CC$-groups}

We next discuss the question, for which infinite cardinals $\kappa$ there exists a saturated pro-$\CC$-group of rank $\kappa$. We will show that the answer is always positive for $\kappa\geq |\kb|$.

We will need some preparatory results. Two epimorphisms $G_1\twoheadrightarrow H$, $G_2\twoheadrightarrow H$ of proalgebraic groups are said to be \emph{isomorphic} if there exists an isomorphism $G_1\to G_2$ such that
$$
\xymatrix{
G_1 \ar[rd] \ar[rr] &  &G_2 \ar[ld] \\
& H &
}
$$
commutes. This definition also applies for $G_1=G_2$.

%\begin{lemma} still needed
%   Let $H$ be a proalgebraic group with $\rank(H)\geq |\overline{k}|$. Then the set of isomorphism classes of epimorphisms $G\twoheadrightarrow H$ with algebraic kernel, has cardinality less or equal to $\rank(H)$.
%\end{lemma}
%\begin{proof}
%   We start by noting that, up to isomorphism, there are at most $|\overline{k}|$ algebraic groups. (This follows from the fact that a polynomial ring in finitely many variables over $k$ has cardinality $\overline{k}$.) It follows similarly, that, for a fixed algebraic group $H''$, there are, up to isomorphism, at most $|\overline{k}|$ epimorphism $H'\twoheadrightarrow H''$, where $H'$ is an algebraic group (allowed to vary).
%   According to Lemma \ref{lemma: rank equals number of coalgebraic subgroups} the proalgebaic group $H$ has at most $\rank(H)$ coalgebraic subgroups. For every coalgebraic subgroup $N$ of $H$, there are, up to isomorphism, at most $|\overline{k}|$ epimorphisms $H'\twoheadrightarrow H''=H/N$.
%
%   Such an epimorphism $H'\twoheadrightarrow H''$ gives rise to an epimorphism $H'\times_{H''}H\twoheadrightarrow H$ with algebraic kernel. Moreover, by Lemma \ref{lemma: every morphism is fibre product}, every epimorphism $G\twoheadrightarrow H$ with algebraic kernel is of this form. Thus there are, up to isomorphism, $\rank(H)|\overline{k}|=\rank(H)$ epimorphisms $G\twoheadrightarrow H$ with algebraic kernel.
%\end{proof}

For a formation $\CC$ we denote by $|\CC|$ the cardinality of the set of all isomorphism classes of $\CC$-groups.

\begin{lemma} \label{lemma: bound isomclasses}
    Let $\kappa$ be an infinite cardinal and let $\CC$ be a formation such that
    \begin{enumerate}
        \item    $|\CC|\leq\kappa$,
        \item  every $\CC$-group has at most $\kappa$ normal closed subgroups and
        \item there are, up to isomorphism, at most $\kappa$ distinct epimorphisms between any two $\CC$-groups.
    \end{enumerate}
Let $H$ be a pro-$\CC$-group with $\rank(H)\leq\kappa$. Then the set of isomorphism classes of epimorphisms $G\twoheadrightarrow H$ of pro-$\CC$-groups with algebraic kernel, has cardinality less or equal to $\kappa$.
\end{lemma}
\begin{proof}
    We start by showing that $H$ has at most $\kappa$ coalgebraic subgroups. Let $\mathcal{M}$ denote the set of coalgebraic subgroups of $H$ and let $\N$ be a neighborhood basis at $1$ for $H$ with $|\N|\leq\kappa$. For every $M\in\mathcal{M}$ there exists an $N\in\N$ with $N\subseteq M$. On the other hand, for a fixed $N$, there are at most $\kappa$ many $M\in\mathcal{M}$ with $N\subseteq M$ by (ii). Thus
    $$|\mathcal{M}|=\left|\bigcup_{N\in\N}\{ M\in\mathcal{M}|\ N\subseteq M \}\right|\leq |\N|\cdot \kappa\leq\kappa. $$
    For a fixed coalgebraic subgroup $N$ of $H$, there are at most $\kappa$ epimorphisms $H'\twoheadrightarrow H/N=H''$ (up to isomorphism) of $\CC$-groups by (i) and (iii).

    Such an epimorphism $H'\twoheadrightarrow H''$ gives rise to an epimorphism $H'\times_{H''}H\twoheadrightarrow H$ of pro-$\CC$-groups with algebraic kernel. Moreover, by Lemma \ref{lemma: every morphism is fibre product}, every epimorphism $G\twoheadrightarrow H$ of pro-$\CC$-groups with algebraic kernel is of this form. Thus there are, up to isomorphism, at most $\kappa\cdot\kappa=\kappa$ epimorphisms $G\twoheadrightarrow H$ of pro-$\CC$-groups with algebraic kernel.
%   As $\rank(H)\leq \kappa$, we may write $H=\varprojlim_{i\in I} H_i$ for some directed set $I$ with $|I|\leq \kappa$ where the transition maps $H_i\twoheadrightarrow H_j$ ($i\geq j$) are epimorphisms of algebraic groups.
\end{proof}

\begin{theo} \label{theo: existence of saturated groups}
 Let $\kappa$ be an infinite cardinal. If $\CC$ is a formation such that
    \begin{enumerate}
    \item    $|\CC|\leq\kappa$,
    \item  every $\CC$-group has at most $\kappa$ normal closed subgroups and
    \item there are, up to isomorphism, at most $\kappa$ distinct epimorphisms between any two $\CC$-groups,
  \end{enumerate}
    then there exists a saturated pro-$\CC$-group of rank $\kappa$. In particular, if $\kappa\geq|\overline{k}|$, there exists a saturated pro-$\CC$-group of rank $\kappa$ for any formation $\CC$.
\end{theo}
\begin{proof}
    Let $\Gamma_0=\prod_i G_i$ be a product of $\CC$-groups such that every $\CC$-isomorphism class occurs exactly $\kappa$ many times in this product. By (i) and Example \ref{ex: rank for fibre product} we have $\rank(\Gamma_0)=\kappa$. According to Lemma~\ref{lemma: bound isomclasses} there are at most $\kappa$ isomorphism classes of epimorphisms $G\twoheadrightarrow \Gamma_0$ of pro-$\CC$-groups with algebraic kernel. Let $\Gamma_1=\prod_{i}(G_i\twoheadrightarrow\Gamma_0)$ be a fibred product of epimorphisms of pro-$\CC$-groups with algebraic kernel, where every isomorphism class of such epimorphisms occurs exactly $\kappa$ times. By Lemma \ref{lemma: bound isomclasses} this fibred product is formed over an index set of cardinality $\kappa$. According to Example \ref{ex: rank for fibre product} we have $\rank(\Gamma_1)\leq\kappa$. Since $\Gamma_0$ is a quotient of $\Gamma_1$ we actually have $\rank(\Gamma_1)=\kappa$. We now keep iterating this construction. So $\Gamma_2$ would be a fibred product of epimorphisms $G_i\twoheadrightarrow \Gamma_1$ and so on. Let $\Gamma=\varprojlim_{i\in\mathbb{N}} \Gamma_i$ be the projective limit of the $\Gamma_i$'s. Then $\Gamma$ is a pro-$\CC$-group of rank $\kappa$.

        We will show that $\Gamma$ is saturated. To do this we will use condition (vii) of Theorem \ref{theo: saturated}. So let
    \begin{equation} \label{eqn: embedding problem proof of existence}
    \xymatrix{
        \Gamma \ar@{->>}^\beta[rd] \ar@{..>}[d] & \\
        G \ar@{->>}^-\alpha[r] & H
    }
    \end{equation}
    be a non-trivial $\CC$-embedding problem. Since $H$ is algebraic, there exists an $i\in\mathbb{N}$ such that $H$ is a quotient of $\Gamma_i$ (i.e., $k[H]$ is contained in $k[\Gamma_i]\subseteq k[\Gamma]$). Then $G\times_H\Gamma_i\twoheadrightarrow \Gamma_i$ is an epimorphism of pro-$\CC$-groups with algebraic kernel. By construction of $\Gamma$, there is an epimorphism $\Gamma\twoheadrightarrow \prod_J (G\times_H\Gamma_i\twoheadrightarrow \Gamma_i)$, where the product is taken over an index set $J$ of cardinality $\kappa$. For $j\in J$ let $\f_j\colon\Gamma\twoheadrightarrow G$ be the composition
    $$\Gamma\twoheadrightarrow \prod_J (G\times_H\Gamma_i\twoheadrightarrow \Gamma_i)\twoheadrightarrow G\times_H \Gamma_i\twoheadrightarrow G,$$
    where the map in the middle is the projection onto the $j$-th factor. Then $\f_j$ is a solution to (\ref{eqn: embedding problem proof of existence}) for every $j\in J$. In fact, since $$\prod_J (G\times_H\Gamma_i\twoheadrightarrow \Gamma_i)=\left(\prod_J(G\twoheadrightarrow H)\right)\times_H\Gamma_i,$$ the family $(\f_j)_{j\in J}$ of solutions is independent.

     Finally, a counting argument shows that
     \begin{enumerate}
        \item Up to isomorphism there are at most $|\overline{k}|$ algebraic groups.
        \item An algebraic group has at most  $|\overline{k}|$ normal closed subgroups.
        \item There are at most $|\overline{k}|$ epimorphisms between any two algebraic groups.
     \end{enumerate}
Thus conditions (i), (ii) and (iii) are satisfied for any $\CC$ if $\kappa\geq |\overline{k}|$.
\end{proof}

We note that the conditions of Theorem \ref{theo: existence of saturated groups} are satisfied for any infinite $\kappa$ if $\CC=\mathtt{C}_k$ is a formation of algebraic groups coming from a formation $\mathtt{C}$ of finite groups. The conditions are also met for any infinite $\kappa$ for the class $\CC$ of all diagonalizable algebraic groups.

\subsection{Uniqueness of saturated pro-$\CC$-groups}

The uniqueness of saturated pro-$\CC$-groups of a fixed rank follows by a rather standard back and forth argument.

\begin{theo} \label{theo: uniqueness}
    Let $\CC$ be a formation and let $\Gamma$ and $\Gamma'$ be saturated pro-$\CC$-groups with $\ir(\Gamma)=\ir(\Gamma')$. Then $\Gamma$ and $\Gamma'$ are isomorphic.
\end{theo}
\begin{proof}
    Let $\kappa=\ir(\Gamma)=\ir(\Gamma)$. We choose $\N$ and $\N'$, neighborhood bases at $1$ of cardinality $\kappa$ for $\Gamma$ and $\Gamma'$ respectively and we well-order them:
    $\N=\{N_\mu|\ \mu<\kappa\}$,\ $\N'=\{N'_\mu|\ \mu<\kappa\}$ with $N_0=\Gamma$ and $N_0'=\Gamma'$.

    We will use transfinite induction to construct descending chains of closed normal subgroups $\{M_\mu|\ \mu<\kappa\}$ and $\{M'_\mu|\ \mu<\kappa\}$ of $\Gamma$ and $\Gamma'$ respectively, and isomorphisms $\f_\mu\colon \Gamma/M_\mu\to \Gamma'/M'_\mu$   such that the following properties are satisfied for every $\mu<\kappa$:
    \begin{enumerate}
        \item The diagram
        $$\xymatrix{
            \Gamma/M_\mu \ar^{\f_\mu}[r] \ar[d] & \Gamma'/M'_\mu \ar[d] \\
            \Gamma/M_\lambda \ar^{\f_\lambda}[r] & \Gamma'/M'_\lambda
        }
        $$
        commutes for $\lambda<\mu$.
        \item $\ir(\Gamma/M_\mu)\leq\mu$.
        \item $M_\lambda\leq N_\mu$ and $M'_\lambda\leq N'_\mu$ for $\lambda <\mu$.
    \end{enumerate}

    For $\mu=0$, we set $M_0=N_0=\Gamma$ and $M_0'=N_0'=\Gamma'$. We also let $\f_0\colon \Gamma/M_0=1\to\Gamma'/M'_0$ be the trivial map.

    Now fix a $\mu<\kappa$ and assume that $M_\lambda,\ M'_\lambda$ and $\theta_\lambda$ have been constructed for all $\lambda<\mu$ such that (i), (ii) and (iii) holds for $\lambda<\mu$ . We have to consider two cases:

    \medskip

    {\bf Case 1:} $\mu=\lambda+1$ is a successor ordinal: Set $P_\mu=N_\lambda\cap M_\lambda$.
    The kernel of $\Gamma/P_\mu\to \Gamma/M_\lambda$ is $M_\lambda/(N_\lambda\cap M_\lambda)\simeq M_\lambda N_\lambda/N_\lambda\hookrightarrow \Gamma/N_\lambda$. Because $\Gamma/N_\lambda$ is algebraic, also the kernel of $\Gamma/P_\mu\to \Gamma/M_\lambda$ is algebraic.
     Since $\Gamma'$ is saturated, the pro-$\CC$-embedding problem $\Gamma/P_\mu\to \Gamma/M_\lambda$, $\Gamma'\to \Gamma'/M'_\lambda\xrightarrow{\f_\lambda^{-1}}\Gamma/M_\lambda$ has a solution, i.e., we can find an epimorphism $\f'\colon\Gamma'\twoheadrightarrow\Gamma/P_\mu$ such that
    $$
    \xymatrix{
        \Gamma' \ar@{->>}_{\f'}[d] \ar@{->>}[dr] & \\
        \Gamma/P_\mu \ar@{->>}[r] & \Gamma/M_\lambda&
    }
    $$
    commutes. For $P'_\mu=\ker(\f')$, this induces an isomorphism
    $\overline{\f'}\colon \Gamma'/P'_\mu\to \Gamma/P_\mu$ such that
    $$
    \xymatrix{
        \Gamma/P_\mu  \ar[d] & \Gamma'/P'_\mu \ar[d] \ar_{\overline{\f'}}[l] \\
        \Gamma/M_\lambda  & \ar_{{\f_\lambda}^{-1}}[l] \Gamma'/M'_\lambda
    }
    $$
    commutes. Set $M'_\mu=P_\mu'\cap N'_\lambda\leq M_\lambda'$. As noted above, the kernel of $\Gamma/P_\mu\to \Gamma/M_\lambda$ is algebraic. Therefore,  by Lemma \ref{lemma: ir for algebraic kernel}, $\ir(\Gamma/P_\mu)=\ir(\Gamma/M_\lambda)\leq \lambda$. If $\lambda$ is finite, $\ir(\Gamma/P_\mu)\leq 1\leq \mu$. If $\lambda$ is infinite, $|\lambda|=|\mu|$  (because $\mu=\lambda+1$). So in either case, $\ir(\Gamma/P_\mu)\leq \mu$. We consider the pro-$\CC$-embedding problem
    $$\Gamma'/M_\mu'\twoheadrightarrow\Gamma'/P_\mu',\ \Gamma\to\Gamma/P_\mu\xrightarrow{\overline{\f'}^{-1}}\Gamma'/P_\mu'$$
    for $\Gamma$. We have $\ir(\Gamma'/P_\mu')=\ir(\Gamma/P_\mu)\leq\mu<\ir(\Gamma)$.
%   Moreover, the kernel of $\Gamma'/M_\mu'\twoheadrightarrow\Gamma'/P_\mu'$ is algebraic.
     As $\Gamma$ is saturated, we can find a solution $\f\colon\Gamma\twoheadrightarrow \Gamma'/M_\mu'$ to the above embedding problem. Set $M_\mu=\ker(\f)\leq P_\mu=N_\lambda\cap M_\lambda$. Then we have a commutative diagram
    $$
    \xymatrix{
        \Gamma/M_\mu \ar^{\f_\mu}[r] \ar[d] & \Gamma'/M_\mu' \ar[d] \\
        \Gamma/P_\mu \ar^{{\overline{\f'}}^{-1}}[r] & \Gamma'/P'_\mu
    }
    $$
    where $\f_\mu=\overline{\f}$ is induced from $\f$.
    Stacking the above two diagrams on top of each other, we find that
    $$
    \xymatrix{
        \Gamma/M_\mu \ar^{\f_\mu}[r] \ar[d]  & \Gamma'/M_\mu' \ar[d] \\
        \Gamma/M_\lambda \ar^{\f_\lambda}[r] & \Gamma'/M_\lambda'
    }
    $$
    commutes. As the kernel of $\Gamma'/M_\mu'\twoheadrightarrow\Gamma'/P_\mu'$ is algebraic, it follows from Lemma \ref{lemma: ir for algebraic kernel} that $$\ir(\Gamma/M_\mu)=\ir(\Gamma'/M_\mu')=\ir(\Gamma'/P_\mu')\leq \mu.$$ So conditions (i), (ii) and (iii) are satisfied for $\mu$.

    \medskip

    {\bf Case 2:} $\mu$ is a limit ordinal: Set $M_\mu=\bigcap_{\lambda<\mu}M_\lambda$ and $M'_\mu=\bigcap_{\lambda<\mu}M'_\lambda$. Then the $\f_\lambda$ for $\lambda<\mu$ induce an isomorphism $\f_\mu\colon \Gamma/M_\mu\to \Gamma'/M_\mu'$. Clearly, conditions (i) and (iii) are satisfied.

    For $\lambda<\mu$ we have $\ir(\Gamma/M_\lambda)\leq|\lambda|\leq|\mu|$. Therefore condition (ii) is satisfied by Lemma \ref{lemma: ir for intersection}. This concludes the construction of the $M_\mu$, $M_\mu'$ and the isomorphisms $\f_\mu$ for $\mu<\kappa=\ir(\Gamma)$.

    \medskip

    It follows from condition (iii) that $\bigcap_{\mu<\kappa} M_\mu=1$ and $\bigcap_{\mu<\kappa} M'_\mu=1$. Therefore the $\f_\mu$ for $\mu<\kappa$ induce an isomorphism $\Gamma\to\Gamma'$.
\end{proof}

%
%\begin{cor}
%   Let $\CC$ be a formation and let $\Gamma$ and $\Gamma'$ be pro-$\CC$-groups of countable rank such that all $\CC$-embedding problems for $\Gamma$ and $\Gamma'$ are solvable. Then $\Gamma$ and $\Gamma'$ are isomorphic.
%\end{cor}

%
%If $\ir(\Gamma)=|\nn|$, then $\Gamma$ is saturated if and only if every algebraic embedding problem for $\Gamma$ is solvable. Thus Proposition \ref{prop: Iwasawa} is a special case of Theorem \ref{theo: uniqueness}.
%

We continue with two examples of saturated pro-$\CC$-groups.

\begin{ex} \label{ex: saturated group for abelian unipotent}
     Let $k$ be a field of characteristic zero and $\CC$ the formation of all abelian unipotent algebraic groups.  (Cf. Example \ref{ex: free abelian unipotent}.) We claim that for any infinite cardinal $\kappa$, the proalgebraic group $\prod_{\kappa}\Ga$ is the saturated pro-$\CC$-group of rank $\kappa$.
\end{ex}
\begin{proof}
    It suffices to show that $\prod_{\kappa}\Ga$ is saturated. Maybe the clearest way to see this it to observe that the functor $V\rightsquigarrow G(V)$ with $G(V)(R)=\Hom_k(V,R)$ for any $k$-algebra $R$, is a (contravariant) equivalence of categories from the category of $k$-vector spaces to the category of pro-$\CC$-groups. We note that $G(V)$ is represented by the symmetric algebra on $V$ and that an epimorphism $G(V)\twoheadrightarrow G(W)$ of pro-$\CC$-groups corresponds to an injection $W\hookrightarrow V$ of $k$-vector spaces. Moreover, for an infinite dimensional $k$-vector space $V$, we have $\rank(G(V))=\dim_k(V)$. Thus the claim reduces to the following obvious statement: If $V$ is a vector space of dimension $\kappa$,  $V_1$ a subspace of $V$ with $\dim_k(V_1)<\kappa$ and $V_1\hookrightarrow V_2$ an embedding of vector spaces with $\dim_k(V_2)\leq\kappa$, then $V_2$ can be embedded into $V$ over $V_1$.
\end{proof}

We next determine the saturated pro-diagonalizable groups.

\begin{ex} \label{ex: saturated prodiagonalizable group}
    Let $\CC$ be the formation of all diagonalizable algebraic groups and let $\kappa$ be an infinite cardinal. According to Theorems \ref{theo: existence of saturated groups} and \ref{theo: uniqueness} there exist a unique saturated pro-$\CC$-group $\Gamma$ of rank $\kappa$. This group $\Gamma$ is of the form $\Gamma=D(M)$ for some abelian group $M$. We would like to determine $M$. According to Theorem \ref{theo: saturated} and Theorem \ref{theo: uniqueness} the group $M$ is uniquely characterized by the following properties:
    \begin{enumerate}
        \item $|M|=\kappa$ and
        \item for every subgroup $M_1$ of $M$ with $|M_1|<|M|$ and embedding of abelian groups $M_1\hookrightarrow M_2$ with $|M_2|\leq |M|$, there exists an embedding of $M_2$ over $M_1$ into $M$.
    \end{enumerate}
To be precise, if $\kappa=\aleph_0$, the statement $|M_1|<|M|$ in (ii) should be replaced by $M_1$ is finitely generated.
This unique abelian group is $M=\bigoplus_{\kappa}\mathbb{Q}\oplus\bigoplus_\kappa\mathbb{Q}/\mathbb{Z}$.\end{ex}
\begin{proof}
Let us first show that $M$ is divisible: Let $a\in M$ and $n\geq 1$, we have to show that there exists $b\in M$ with $nb=a$. Let $M_1=\langle a\rangle$ denote the cyclic subgroup of $M$ generated by $a$. Then $M_1$ is contained in a cyclic abelian group $M_2$ such that $nM_2=M_1$. (If $a$ has infinite order consider the inclusion $\mathbb{Z}n\subseteq \mathbb{Z}$ and if $a$ has finite order $m$, consider the inclusion image $\mathbb{Z}n/\mathbb{Z}mn\subseteq\mathbb{Z}/\mathbb{Z}mn$.) Since $M_2$ embeds into $M$ over $M_1$ it follows that $M$ is divisible.

Any divisible group is of the form $\bigoplus_{\kappa_0}\mathbb{Q}\bigoplus_{p\in\mathbb{P}} \oplus_{\kappa_p}\mathbb{Z}(p^\infty)$, where $\mathbb{P}$ denotes the set of all prime numbers and $\mathbb{Z}(p^\infty)$ the $p$-Pr\"{u}fer group. (See e.g., \cite[Theorem 4]{Kaplansky:InfiniteAbelianGroups}.) The cardinal numbers $\kappa_0$ and $\kappa_p$ are uniquely determined. The groups $\bigoplus_\kappa\mathbb{Q}$ and $\bigoplus_\kappa \mathbb{Z}(p^\infty)$ have cardinality $\kappa$ and so embed into $M$ by (ii). It follows that $M=\bigoplus_{\kappa}\mathbb{Q}\bigoplus_{p\in\mathbb{P}} \oplus_{\kappa}\mathbb{Z}(p^\infty)$. As $\bigoplus_{p\in\mathbb{P}}\mathbb{Z}(p^\infty)=\mathbb{Q}/\mathbb{Z}$ the claim follows.
\end{proof}

\subsection{Free pro-$\CC$-groups and saturation}

Having discussed free pro-$\CC$-groups and saturated pro-$\CC$-groups, we now address the question under which circumstances these notions coincide. Given the uniqueness of saturated pro-$\CC$-groups of a fixed rank (Theorem \ref{theo: uniqueness}) the question boils down to, whether or not a free pro-$\CC$-group is saturated.

\begin{prop} \label{prop: free groups are almost saturated}
    Let $\CC$ be a formation
    % such that for every $\CC$-group $G$ there exists a finite subset $B$ of $G(\overline{k})$ with $G=\langle B\rangle$. (According to  this condition is always satisfied in characteristic zero.)
    and let $X$ be an infinite set. Let
    \begin{equation} \label{eqn: embedding problem for free group}
    \xymatrix{
        \Gamma^\CC(X) \ar@{->>}^\beta[rd] \ar@{..>>}[d]& \\
        G \ar@{->>}^\alpha[r] & H
    }
    \end{equation}
    be a pro-$\CC$-embedding problem for $\Gamma^\CC(X)$ with algebraic kernel and $\rank(H)<|X|$. Assume that there exists a finite set $B\subseteq\ker(\alpha)(\overline{k})$ with $\ker(\alpha)=\langle B\rangle$. Then (\ref{eqn: embedding problem for free group}) is solvable.
\end{prop}
\begin{proof}
    According to Lemma \ref{lemma: every morphism is fibre product} there exist algebraic groups $H', H''$ and epimorphisms $H'\twoheadrightarrow H''$ and $H\twoheadrightarrow H''$ such that $G=H'\times_{H''}H$.

    We will define a map $\varphi\colon X\to G(\overline{k})$. Let us first define $\varphi(x)$ for $\iota(x)\notin\ker(\beta)(\overline{k})$. If $x$ maps to the identity under $\Gamma^\CC(X)\twoheadrightarrow H\twoheadrightarrow H''$, we set
    $\varphi(x)=(1,\beta(\iota(x)))\in (H'\times_{H''}H)(\overline{k})$. There are only finitely many elements, say $x_1,\ldots,x_n$, of $X$ that do not map to the identity under $\Gamma^\CC(X)\twoheadrightarrow H''$. For these we can choose $h'_i\in H'(\overline{k})$ such that the image of $h'_i$ in $H''(\overline{k})$ agrees with the image of $x_i$ in $H''(\overline{k})$. We set $\varphi(x_i)=(h_i',\beta(\iota(x_i)))\in (H'\times_{H''}H)(\overline{k})$ for $i=1,\ldots,n$.

    Now let us define $\varphi(x)$ for $\iota(x)\in\ker(\beta)(\overline{k})$.
    By assumption there exists a finite subset $B=\{b_1,\ldots,b_m\}$ of $\ker(\alpha)(\overline{k})$ with $\ker(\alpha)=\langle B \rangle$.
    It follows from Lemma \ref{lemma: ker has infinitely many elements} that $X\cap\iota^{-1}(\ker(\beta)(\overline{k}))$ is infinite. So we can find distinct elements $x_1',\ldots,x_m'\in X\cap\ker(\beta)(\overline{k})$. We set $\varphi(x_i')=b_i$ for $i=1,\ldots,m$. For $x\notin \{x_1',\ldots,x_m'\}$ we set $\varphi(x)=1$. This concludes the definition of $\varphi$. Note that by construction $\alpha\circ \varphi=\beta\circ \iota$ on $X$.
    We claim that $\varphi\colon X\to G(\overline{k})$ converges to $1$.

    If $N$ is a coalgebraic subgroup of $H$ contained in $\ker(H\twoheadrightarrow H'')$, then $1\times_{H''}N$ is a coalgebraic subgroup of $G=H'\times_{H''}H$. Moreover, every coalgebraic subgroup $N'$ of $G$ contains a coalgebraic subgroup of this form. (Namely, the intersection of $N'$ with the kernel of the projection $H'\times_{H''} H\twoheadrightarrow H'$.)
    So it suffices to see that $\varphi$ maps almost all elements of $X$ into $(1\times_{H''}N)(\kb)$, where $N$ is a coalgebraic subgroup of $H$ contained in $\ker(H\twoheadrightarrow H'')$. This however, is clear from the definition of $\varphi$.
%
%   if $N'$ is a coalgebraic subgroup of $G=H'\times_{H''}H$, then the image of $N'$ under the projection $H'\times_{H''}H\twoheadrightarrow H$ intersected with $\ker(H\twoheadrightarrow H'')$ is a coalgebraic subgroup $N$ of $H$ with $1\times_{H''} N\subseteq N'$.
%   If $N$ is a coalgebraic subgroup of $G$, then there is a coalgebraic subgroup $M$ of $H$ such that
%
%
%    $\alpha^{-1}(N)$ is a coalgebraic subgroup
%

    We next show that $\langle \varphi(X)\rangle=G$. Using Lemma \ref{lemma: morphisms and generation} we see that $\alpha(\langle \varphi(X)\rangle)=\langle\alpha(\varphi(X))\rangle=\langle\beta(\iota(X)\rangle=H$. Since $\ker(\alpha)\subseteq \langle \varphi(X)\rangle$, this implies $\langle \varphi(X)\rangle=G$.\

    Thus $\varphi$ induces an epimorphism $\f\colon\Gamma^\CC(X)\to G$. Clearly $\f$ is a solution of (\ref{eqn: embedding problem for free group}).
\end{proof}

\begin{theo} \label{theo: free=saturated}
    Let $k$ be a field of characteristic zero, $\CC$ a formation and $\Gamma$ a pro-$\CC$-group with $\rank(\Gamma)\geq|k|$. Then $\Gamma$ is isomorphic to a free pro-$\CC$-group on a set of cardinality $\rank(\Gamma)$ if and only if $\Gamma$ is saturated.
\end{theo}
\begin{proof}
    Let $X$ be a set with $|X|=\rank(\Gamma)$. Assume first that $\Gamma\simeq \Gamma^\CC(X)$. Then $\rank(\Gamma^\CC(X))=|X|$ and because in characteristic zero any algebraic group is finitely generated (Lemma \ref{lemma: finite generation}), it follows from Proposition \ref{prop: free groups are almost saturated} that $\Gamma\simeq \Gamma^\CC(X)$ is saturated.

    Conversely, assume $\Gamma$ is saturated. We know from Corollary \ref{cor: rank of free is X} that $\rank(\Gamma^\CC(X))=|X|$. So, as above, it follows from Proposition \ref{prop: free groups are almost saturated} that $\Gamma^\CC(X)$ is saturated. Since $\rank(\Gamma^\CC(X))=\rank(\Gamma)$ we can conclude that $\Gamma^\CC(X)$ and $\Gamma$ are isomorphic by Theorem \ref{theo: uniqueness}.
\end{proof}

\begin{cor} \label{cor: Iwasawa for algebaic groups}
    Let $k$ be a countable field of characteristic zero and let $\CC$ be a formation. Moreover, let $X$ be a countably infinite set and $\Gamma$ a pro-$\CC$-group of countable rank. Then $\Gamma$ is isomorphic to the free pro-$\CC$-group $\Gamma^\CC(X)$ if and only if every $\CC$-embedding problem for $\Gamma$ is solvable.
\end{cor}
\begin{proof}
By Corollary \ref{cor: rank of free is X} the free pro-$\CC$-group $\Gamma^\CC(X)$ has rank $|X|$. By Theorem \ref{theo: free=saturated} $\Gamma\simeq \Gamma^\CC(X)$ is saturated. Since $\rank(\Gamma)$ is countable, condition (ii) of Theorem \ref{theo: saturated} reduces to the statement that all $\CC$-embedding problems are solvable.

Conversely, if every $\CC$-embedding problem for $\Gamma$ is solvable, then $\rank(\Gamma)$ has to be infinite. Since $\rank(\Gamma)$ is countable, we see that $\Gamma$ is saturated. Thus $\Gamma\simeq \Gamma^\CC(X)$ by Theorem \ref{theo: free=saturated}.
\end{proof}

Theorem \ref{theo: free=saturated} is not true in positive characteristic. The corresponding statement in positive characteristic fails for two reasons. Firstly, a free proalgebaic group over a perfect field is always reduced (Remark \ref{rem: free group is reduced}) while a saturated pro-$\CC$-group need not be reduced. Indeed, if $\CC$ contains at least one algebraic group that is not reduced, then a saturated pro-$\CC$-group is not reduced.

Secondly, and maybe more fundamentally, not all algebraic groups in positive characteristic are finitely generated, so, if $\CC$ contains an algebraic group that is not finitely generated, then it cannot be a quotient of a free pro-$\CC$-group and therefore the free pro-$\CC$-group cannot be saturated.

We conclude this section with two examples that illustrate Theorem \ref{theo: free=saturated}. These examples also show that Theorem \ref{theo: free=saturated} is not true without the assumption that $\rank(\Gamma)\geq |k|$.

\begin{ex}
    Let $k$ be a field of characteristic zero, $\CC$ the formation of all abelian unipotent groups and $\kappa$ an infinite cardinal. According to Examples \ref{ex: saturated group for abelian unipotent} and \ref{ex: free abelian unipotent}, the saturated pro-$\CC$-group of rank $\kappa$ is $\prod_\kappa \Ga$, while the free pro-$\CC$-group on a set of cardinality $\kappa$ is $\prod_{Y}\Ga$, where $|Y|=\kappa\cdot \dim_k(\kb)$. These groups are isomorphic if and only if $\dim_k(\kb)\leq\kappa$. As predicted by Theorem \ref{theo: free=saturated}, this is the case if $\kappa\geq|k|$.
\end{ex}

\begin{ex}
    Let $\CC$ be the formation of all diagonalizable algebraic groups and let $X$ be an infinite set. We saw in Example \ref{ex: free prodiagonalizable group} that $\Gamma^\CC(X)=D(\bigoplus_X \overline{k}^\times)$. On the other hand, we saw in Example \ref{ex: saturated prodiagonalizable group}
    that the saturated pro $\CC$-group of rank $|X|$ is given by $D(\bigoplus_{X}\mathbb{Q}\oplus\bigoplus_X\mathbb{Q}/\mathbb{Z})$.    In positive characteristic $p$ these groups are not isomorphic since $\bigoplus_X \overline{k}^\times$ has no $p$-torsion, whereas $\bigoplus_{X}\mathbb{Q}\oplus\bigoplus_X\mathbb{Q}/\mathbb{Z}$ does.
    So let us assume that $k$ has characteristic zero. The group  $\bigoplus_{X}\mathbb{Q}\oplus\bigoplus_X\mathbb{Q}/\mathbb{Z}$ has cardinality $|X|$, whereas the group $\bigoplus_X \overline{k}^\times$ has cardinality $|X|\cdot|k|$. Thus if $|X|<|k|$, the groups are not isomorphic. On the other hand, if $|X|\geq |k|$, the groups must be isomorphic by Theorem \ref{theo: free=saturated}. An explicit (non-canonical) isomorphism between the two groups can be constructed by observing that $\overline{k}^\times\simeq\mathbb{Q}/\mathbb{Z}\oplus_\kappa\mathbb{Q}$ for some cardinal $\kappa\leq|k|$.
\end{ex}

%%%%%%%%%%%%%%%%%%%%%
% References
%%%%%%%%%%%%%%%%%%%%%

\bibliographystyle{alpha}

\end{document}